\documentclass{amsart}

\usepackage{bbm} 
\usepackage{amssymb} 
\usepackage{hyperref}

\newcommand*{\mailto}[1]{\href{mailto:#1}{\nolinkurl{#1}}}
\newcommand{\arxiv}[1]{\href{http://arxiv.org/abs/#1}{arXiv:#1}}

\newtheorem{theorem}{Theorem}[section]

\newtheorem{lemma}[theorem]{Lemma}
\newtheorem{proposition}[theorem]{Proposition}
\newtheorem{corollary}[theorem]{Corollary}
\newtheorem{hypothesis}[theorem]{Hypothesis}
\newtheorem{remark}[theorem]{Remark}

\newcommand{\R}{{\mathbb R}}
\newcommand{\N}{{\mathbb N}}

\newcommand{\C}{{\mathbb C}}

\newcommand{\dbspr}[2]{[ #1 , #2 ]}
\newcommand{\spr}[2]{\langle #1 , #2 \rangle}
\DeclareMathOperator{\linspan}{span}

\newcommand{\E}{\mathrm{e}}
\newcommand{\I}{\mathrm{i}}

\newcommand{\im}{\mathrm{Im}}
\newcommand{\re}{\mathrm{Re}}
\newcommand{\dom}[1]{\mathrm{dom}\left(#1\right)}
\newcommand{\mul}[1]{\mathrm{mul}\left(#1\right)}
\DeclareMathOperator{\ran}{ran}

\newcommand{\Tpre}{T_0}
\newcommand{\Tmin}{T_{\mathrm{min}}}
\newcommand{\Tmax}{T_{\mathrm{max}}}
\newcommand{\Tloc}{T_{\mathrm{loc}}}

\newcommand{\Llocr}{L^1_{\mathrm{loc}}((a,b);\varrho)}
\newcommand{\AClocp}{AC_{\mathrm{loc}}((a,b);\varsigma)}
\newcommand{\AClocr}{AC_{\mathrm{loc}}((a,b);\varrho)}
\newcommand{\Hab}{H^1(a,b)}
\newcommand{\Lrtmu}{L^2(\R^\times;\mu)}
\newcommand{\Lrmu}{L^2(\R;\mu)}
\newcommand{\qd}{{[1]}}
\newcommand{\Deftau}{\D_\tau}
\newcommand{\D}{\mathfrak{D}}
\newcommand{\M}{\mathrm{M}}

\DeclareMathOperator{\supp}{supp}

\newcommand{\oo}{o}

\newcommand{\eps}{\varepsilon}

\numberwithin{equation}{section}

\begin{document}

\title[Spectral theory of singular left-definite operators]{Direct and inverse spectral theory of singular left-definite Sturm--Liouville operators}

\author[J.\ Eckhardt]{Jonathan Eckhardt}
\address{Faculty of Mathematics\\ University of Vienna\\
Nordbergstrasse 15\\ 1090 Wien\\ Austria}
\email{\mailto{jonathan.eckhardt@univie.ac.at}}
\urladdr{\url{http://homepage.univie.ac.at/jonathan.eckhardt/}}

\thanks{J.\ Differential Equations {\bf 253} (2012), no.\ 2, 604--634}
\thanks{{\it Research supported by the Austrian Science Fund (FWF) under Grant No.\ Y330}}

\keywords{Sturm--Liouville theory, left-definite problems, spectral theory}
\subjclass[2010]{Primary 34L05, 34B20; Secondary 46E22, 34B40}

\begin{abstract}
We discuss direct and inverse spectral theory of self-adjoint Sturm--Liouville relations with separated boundary conditions in the left-definite setting. In particular, we develop singular Weyl--Titchmarsh theory for these relations. 
 Consequently, we apply de Branges' subspace ordering theorem to obtain inverse uniqueness results for the associated spectral measure.
 The results can be applied to solve the inverse spectral problem associated with the Camassa--Holm equation.
\end{abstract}

\maketitle

\section{Introduction}

 Consider the left-definite Sturm--Liouville problem
 \begin{align}\label{eqntauLD}
  -\frac{d}{dx} \left(p(x) \frac{d}{dx}y(x)\right) +q(x) y(x) = zr(x) y(x)
\end{align}
 on some interval $(a,b)$. Here, by left-definite we mean that the real-valued function $r$ is allowed to change sign but $p$ and $q$ are assumed to be non-negative.
In the case of a regular left endpoint, Bennewitz \cite{ben3}, Brown and Weikard \cite{bebrwe} recently developed   Weyl--Titchmarsh theory for such problems, analogously to the right-definite case. Moreover, they were also able to prove that the associated spectral measure uniquely determines the left-definite Sturm--Liouville problem up to a so-called Liouville transform.

In the present paper we give an alternative proof of this result, using de Branges' subspace ordering theorem for Hilbert spaces of entire functions.  
In fact, this approach allows us to deal with a larger class of problems. For instance, we allow the left endpoint to be quite singular and the weight function $r$ to be a real-valued Borel measure. 
However, at a second glance our approach is not too different from the approach taken in \cite{ben3} and~\cite{bebrwe}. The authors there prove Paley--Wiener type results to describe the spectral transforms of functions with compact support in order to obtain an appropriate Liouville transform. We will show that these spaces of transforms are actually hyperplanes in some de Branges spaces associated with our left-definite Sturm--Liouville problem. This will allow us to apply de Branges' subspace ordering theorem to obtain a suitable Liouville transform. 

As in~\cite{ben3} and~\cite{bebrwe}, the main motivation for this work is the Camassa--Holm equation, an integrable, non-linear wave equation which models unidirectional propagation of waves on shallow water.
 Due to its many remarkable properties, this equation has gotten a lot of attention recently and we only refer to e.g.\ \cite{beco}, \cite{const}, \cite{coes}, \cite{como}, \cite{hora} for further information.
 Associated with the Camassa--Holm equation is the left-definite Sturm--Liouville problem 
\begin{align}\label{eqnIsospec}
 - y''(x) + \frac{1}{4} y(x) = z \omega(x) y(x) 
\end{align}
 on the real line.  
 Direct, and in particular inverse spectral theory of this weighted Sturm--Liouville problem are of peculiar interest for solving the Camassa--Holm equation. 
 Provided $\omega$ is strictly positive and smooth enough, it is possible to transform this problem into a Sturm--Liouville problem in potential form and some inverse spectral theory may be drawn from this. However, in order to incorporate the main interesting phenomena (wave breaking \cite{coes} and multi-peakon solutions \cite{bss}, \cite{cost}) of the dispersionless Camassa--Holm equation, it is necessary to allow $\omega$  at least to be an arbitrary finite signed Borel measure on $\R$.
  In~\cite{ben3}, \cite{bebrwe} the authors were able to prove an inverse uniqueness result under some restrictions on the measure $\omega$, which for example prohibits the case of multi-peakon solutions of the Camassa--Holm equation. 
 Using the results of the present paper we are able to avoid these restrictions and to cover the case of arbitrary real finite measures $\omega$; see~\cite{IsospecCH}.
 
 Note that this application also requires us to consider our Sturm--Liouville problem~\eqref{eqntauLD} with measure coefficients. For further information on measure Sturm--Liouville equations see e.g.\ \cite{ben2} or \cite{measureSL} and the references therein. 
  Moreover, the fact that we allow the weight measure to vanish on arbitrary sets, makes it necessary to work with linear relations instead of operators. Regarding the notion of linear relations, we refer to e.g.\ \cite{arens}, \cite{cross}, \cite{dijksnoo}, \cite{dijksnoo2}, \cite{haase} or  for a brief review, containing all facts which are needed here \cite[Appendix~B]{measureSL}.
 
The paper is organized as follows. After some preliminaries about left-definite measure Sturm--Liouville equations, we give a review of linear relations associated with~\eqref{eqntauLD} in a modified Sobolev space. In particular, we discuss self-adjoint realizations with separated boundary conditions in Section~\ref{secSA}.
 Since a lot of this first part are minor generalizations of results in e.g.\ \cite{benbro}, \cite{bebrwe}, \cite{measureSL}, we will omit most of the proofs. 
 In the consecutive two sections we develop Weyl--Titchmarsh theory for such self-adjoint relations. 
 This part can essentially be done along the lines of the singular Weyl--Titchmarsh theory, recently introduced in \cite{geszin} and \cite{kst2} for Schr\"{o}dinger operators. 
 Section~\ref{secLDdB} introduces some de Branges spaces associated with a left-definite self-adjoint Sturm--Liouville problem. Moreover, we provide crucial properties of these spaces, which are needed for the proof of our inverse uniqueness result, which is carried out in the last section. In particular, this last section provides an inverse uniqueness result, which applies to the isospectral problem of the Camassa--Holm equation.
 Finally, in the appendix we give a brief review of de Branges' theory of Hilbert spaces of entire functions as far as it is needed. 
 
 Before we start, let us recall some facts about functions which are absolutely continuous with respect to some measure.
 Therefore let $(a,b)$ be an arbitrary interval and $\mu$ be a complex-valued Borel measure on $(a,b)$.
With $AC_{\mathrm{loc}}((a,b);\mu)$ we denote the set of all left-continuous functions, which are locally absolutely continuous with respect to the measure $\mu$. 
These are precisely the functions $f$ which can be written in the form
\begin{align*}
 f(x) = f(c) + \int_c^x h(s) d\mu(s), \quad x\in(a,b)
\end{align*}
for some $h\in L^1_{\mathrm{loc}}((a,b); \mu)$, where the integral has to be read as
\begin{align*}
 \int_c^x h(s) d\mu(s) = \begin{cases}
                           \int_{[c,x)} h(s)d\mu(s),             & \text{if }x>c, \\
                           0, & \text{if }x=c, \\
                           - \int_{[x,c)} h(s)d\mu(s),             & \text{if }x<c.
                         \end{cases}
\end{align*}
The function $h$ is the Radon--Nikod\'{y}m derivative of $f$ with respect to $\mu$. It is uniquely defined in $L^1_{\mathrm{loc}}((a,b); \mu)$ and we write $d f/d\mu = h$. Every function $f$ which is locally absolutely continuous with respect to $\mu$ is locally of bounded variation and hence also its right-hand limits 
\begin{align*}
 f(x+) = \lim_{\eps\downarrow 0} f(x+\eps), \quad x\in(a,b)
 \end{align*}
 exist everywhere and may differ from $f(x)$ only if  $\mu$ has mass in $x$. 
 Furthermore, we will repeatedly use the following integration by parts formula for complex-valued Borel measures $\mu$, $\nu$ on $(a,b)$  (see e.g.\ \cite[Theorem~21.67]{hesr})
\begin{align}\label{eqnLDpartint}
\int_{\alpha}^\beta F(x)d\nu(x) = \left. FG\right|_\alpha^\beta - \int_{\alpha}^\beta  G(x+)d\mu(x), \quad \alpha,\,\beta\in(a,b),
\end{align}
where $F$, $G$ are left-continuous distribution functions of $\mu$, $\nu$ respectively.

\section{Left-definite Sturm--Liouville equations}

Let $(a,b)$ be an arbitrary interval and $\varrho$, $\varsigma$ and $\chi$ be complex-valued Borel measures 
 on $(a,b)$. We are going to define a linear differential expression $\tau$ which is informally given by
\begin{align*}
\tau f =  \frac{d}{d\varrho} \left(-\frac{df}{d\varsigma} + \int fd\chi \right).
\end{align*}
In the rest of this paper we will always assume that our measures satisfy the following four properties.

\begin{hypothesis} \quad
\begin{enumerate}\label{genhypLD}
 \item The measure $\varrho$ is real-valued.
 \item The measure $\varsigma$ is positive and supported on the whole interval.
 \item The measure $\chi$ is positive but not identically zero.
 \item\label{hypgenLDsep} The measure $\varsigma$ has no point masses in common with $\chi$ or $\varrho$, i.e.
  \begin{align*}
    \varsigma(\lbrace x\rbrace) \chi(\lbrace x\rbrace) = \varsigma(\lbrace x\rbrace) \varrho(\lbrace x\rbrace) = 0, \quad x\in(a,b).
  \end{align*}
\end{enumerate}
\end{hypothesis}

The maximal domain $\Deftau$ of functions such that the expression $\tau f$ makes sense consists of all functions 
$f\in \AClocp$ for which the function
\begin{align}\label{eqnLDtauqdptheta}
 -\frac{df}{d\varsigma}(x) + \int_c^x fd\chi, \quad x\in(a,b)
\end{align}
is locally absolutely continuous with respect to $\varrho$, i.e.\ there is some representative of this function 
lying in $\AClocr$. As a consequence of the assumption on the support of $\varsigma$, this representative is unique. 
We then set $\tau f\in\Llocr$ to be the Radon--Nikod\'{y}m derivative of this function with respect to $\varrho$.
One easily sees that this definition is independent of $c\in(a,b)$ since the corresponding 
functions~\eqref{eqnLDtauqdptheta} as well as their unique representatives only differ by an additive constant.
As usual, the Radon--Nikod\'{y}m derivative with respect to $\varsigma$ of some $f\in\Deftau$ is denoted with $f^\qd$ and referred to as the quasi-derivative of $f$.

Note that this definition includes classical Sturm--Liouville and Jacobi expressions as special cases. 
 The following existence and uniqueness theorem for solutions of measure Sturm--Liouville equations may be found in~\cite[Theorem~3.1]{measureSL}.

\begin{theorem}\label{thmLDexisuniq}
For each $g\in\Llocr$, $c\in(a,b)$, $d_1$, $d_2\in\C$ and $z\in\C$ there is a unique solution of the initial value problem 
\begin{align*}
 (\tau - z)f=g \quad\text{with}\quad f(c)=d_1 \quad\text{and}\quad f^\qd(c) = d_2.
\end{align*}
\end{theorem}

We say that $\tau$ is regular at an endpoint if the measures $\varrho$, $\varsigma$ and $\chi$ are finite near this endpoint.
 In this case, if $g$ is also integrable near this endpoint, then each solution of the equation $(\tau-z)f=g$ may be continuously extended to this endpoint. Moreover, the initial point $c$ in Theorem~\ref{thmLDexisuniq} may be chosen as this endpoint (see e.g.\ \cite[Theorem~3.5]{measureSL}).

Associated with our differential expression $\tau$ is a linear relation $\Tloc$ in the space $\AClocp$ defined by
\begin{align*}
 \Tloc = \lbrace (f,f_\tau)\in \AClocp^2 \,|\, f\in\Deftau,\, \tau f=f_\tau \text{ in }\Llocr \rbrace.
\end{align*}
Regarding notation we will make the following convention.
Given some pair $f\in\Tloc$ we will denote its first component also with $f$ and the second component with $f_\tau$.
Moreover, if $g\in\AClocp$ and $f$ is a solution of $(\tau-z)f=g$ for some $z\in\C$, then this solution $f$ will often be identified with the pair $(f,g+zf)\in\Tloc$.

In the right-definite theory, a crucial role is played by the Wronskian of two functions and the associated Lagrange identity.
The corresponding quantity in the left-definite case is the function
\begin{align}\label{eqnModWron}
 V(f,g^\ast)(x) = f_\tau(x) g^\qd(x)^\ast - f^\qd(x) g_\tau(x)^\ast, \quad x\in(a,b),
\end{align}
defined for all pairs $f$, $g\in\Tloc$. 
Using the integration by parts formula~\eqref{eqnLDpartint} and property~\eqref{hypgenLDsep} in Hypothesis~\ref{genhypLD} one obtains the following Lagrange identity for this modified Wronskian.

\begin{proposition}\label{propLDLagrange}
For every $f$, $g\in\Tloc$ and $\alpha$, $\beta\in(a,b)$ we have
\begin{align*}
 V(f,g^\ast)(\beta) - V(f,g^\ast)(\alpha) & = \int_\alpha^\beta f_\tau(x) g(x)^\ast - f(x) g_\tau(x)^\ast d\chi(x) \\
                                          & \quad\quad\quad + \int_\alpha^\beta f_\tau^\qd(x) g^\qd(x)^\ast - f^\qd(x) g_\tau^\qd(x)^\ast d\varsigma(x).
\end{align*}
\end{proposition}


As a consequence of this Lagrange identity one sees that for each $z\in\C$ the modified Wronskian $V(u_1,u_2)$ of two solutions $u_1$, $u_2$ of $(\tau-z)u=0$ is constant. Furthermore, for $z\not=0$ the solutions $u_1$, $u_2$ are linearly dependent if and only if $V(u_1,u_2)=0$.
 Another useful identity for the modified Wronskian is the following Pl\"{u}cker identity, which follows similarly as in \cite[Proposition~3.4]{measureSL}.

\begin{proposition}\label{propLDplucker}
 For every $f_1$, $f_2$, $f_3$, $f_4\in\Tloc$ we have
 \begin{align*}
  0 & = V(f_1,f_2)V(f_3,f_4) + V(f_1,f_3)V(f_4,f_2) + V(f_1,f_4)V(f_2,f_3).
 \end{align*}
\end{proposition}


In order to obtain a linear relation in a Hilbert space, we introduce a modified Sobolev space $\Hab$.
It consists of all functions $f$ on $(a,b)$ which are locally absolutely continuous with respect to $\varsigma$ such that $f$ is square integrable with respect to $\chi$ and the Radon--Nikod\'{y}m derivative $df/d\varsigma$ is square integrable with respect to $\varsigma$.
 The space $\Hab$ is equipped with the inner product
 \begin{align*}
  \spr{f}{g} = \int_a^b f(x) g(x)^\ast d\chi(x) + \int_a^b f^\qd(x) g^\qd(x)^\ast d\varsigma(x), \quad f,\,g\in \Hab.
 \end{align*}
 Hereby note that $f$ and $g$ are always continuous in points of mass of $\chi$ in virtue of property~\eqref{hypgenLDsep} in Hypothesis~\ref{genhypLD}. It is not surprising that this modified Sobolev space turns out to be a reproducing kernel Hilbert space (see e.g.\ \cite[Section~2]{bebrwe}).

In order to obtain the maximal relation $\Tmax$ in $\Hab$ associated with our differential expression $\tau$ we restrict $\Tloc$ by
\begin{align*}
 \Tmax = \left\lbrace (f, f_\tau)\in\Hab\times\Hab \,|\, (f,f_\tau)\in\Tloc \right\rbrace.
\end{align*}
 The following characterization of $\Tmax$ as weak solutions of our differential equation will be quite useful (the proof can be done along the lines of \cite[Proposition~2.4]{bebrwe}). 

\begin{proposition}\label{propLDweakform}
Some $(f,f_\tau)\in \Hab\times \Hab$ lies in $\Tmax$ if and only if
\begin{align}\label{eqnLDweakformint}
 \int_a^b f(x)g(x)^\ast d\chi(x) + \int_a^b f^\qd(x) g^\qd(x)^\ast d\varsigma(x) = \int_a^b f_\tau(x) g(x)^\ast d\varrho(x)
\end{align}
for each $g\in H_c^1(a,b)$.
\end{proposition}


Here, $H_c^1(a,b)$ denotes the linear subspace of $\Hab$ consisting of all functions with compact support.
 Consequently, some function $h\in \Hab$ lies in the multi-valued part of $\Tmax$ if and only if $h=0$ almost everywhere with respect to $|\varrho|$.


We say some function $f\in\AClocp$ lies in $\Hab$ near an endpoint if $f$ is square integrable with respect to $\chi$ near this endpoint and its quasi-derivative is square integrable with respect to $\varsigma$ near this endpoint. Furthermore, we say some pair $f\in\Tloc$ lies in $\Tmax$ near an endpoint if both components $f$ and $f_\tau$ lie in $\Hab$ near this endpoint. Clearly, some $f\in\Tloc$ lies in $\Tmax$ if and only if it lies in $\Tmax$ near $a$ and near $b$. Using the Lagrange identity one shows the following properties of the modified Wronskian on $\Tmax$.

\begin{lemma}\label{lemLDlagrange}
 If $f$ and $g$ lie in $\Tmax$ near an endpoint, then the limit of $V(f,g^\ast)(x)$ as $x$ tends to this endpoint
 exists and is finite. If $f$ and $g$ even lie in $\Tmax$, then
 \begin{align}\label{eqnLDLagrangeSP}
  \spr{f_\tau}{g} - \spr{f}{g_\tau} = V(f,g^\ast)(b) - V(f,g^\ast)(a). 
 \end{align}
 Moreover, $V(\,\cdot\,,\cdot\,)(a)$ and $V(\,\cdot\,,\cdot\,)(b)$ are continuous bilinear forms on $\Tmax$ with respect to the product topology on $\Tmax$.
\end{lemma}

If $\tau$ is regular at an endpoint, then it is not hard to see that for each $f$ which lies in $\Tmax$ near this endpoint, the limits of $f(x)$, $f^\qd(x)$ and $f_\tau(x)$ as $x$ tends to this endpoint 
exist and are finite. 
 Of course in this case, equation~\eqref{eqnModWron} extends to this regular endpoint provided that $f$ and $g$ lie in $\Tmax$ near this endpoint. 

Next we will collect some more properties of the modified Sobolev space $\Hab$ and the maximal relation $\Tmax$. 
The next proposition may be proved similarly to \cite[Theorem~2.6]{bebrwe} and \cite[Proposition~2.7]{bebrwe}.
 Here and in the following, $H_0^1(a,b)$ will denote the closure of $H_c^1(a,b)$ in $\Hab$.

\begin{proposition}\label{propLDdecompH1}
We have $\Hab=H_0^1(a,b)\oplus \ker(\Tmax)$, with
\begin{align}
 \dim\ker(\Tmax) = \begin{cases}
     0, & \text{if }\varsigma+\chi\text{ is infinite near both endpoints,} \\
     1, & \text{if }\varsigma+\chi\text{ is finite near precisely one endpoint,} \\
     2, & \text{if }\varsigma+\chi\text{ is finite.} 
                 \end{cases}
\end{align}
 Moreover, there are (up to scalar multiples) unique non-trivial real solutions $w_a$, $w_b$ of $\tau u=0$ which lie in $\Hab$ near $a$, $b$ respectively and satisfy
  \begin{align}\label{eqnLDwaneara}
   \lim_{\alpha\rightarrow a} g(\alpha) w_a^\qd(\alpha) = \lim_{\beta\rightarrow b} g(\beta) w_b^\qd(\beta) =0, \quad g\in\Hab.
  \end{align}
  The solutions $w_a$ and $w_b$ are linearly independent.
\end{proposition}

Also note that the functions 
\begin{align}\label{eqnKerTmaxInc}
x \mapsto w_a(x) w_a^\qd(x) \quad\text{and}\quad x\mapsto w_b(x) w^\qd_b(x)
\end{align}
 are increasing on $(a,b)$ and strictly positive and negative, respectively. 
Now for each fixed $c\in(a,b)$ we introduce the function
\begin{align}
 \delta_c(x) = \frac{1}{W(w_b,w_a)} \begin{cases}
                 w_a(x)w_b(c), & \text{if } x\in(a,c], \\
                 w_a(c)w_b(x), & \text{if } x\in(c,b),
               \end{cases}
\end{align}
with the usual Wronskian of $w_a$ and $w_b$ 
\begin{align*}
 W(w_b,w_a) = w_b(x)w_a^\qd(x) - w_b^\qd(x) w_a(x),
\end{align*}
where the right-hand side is independent of $x\in(a,b)$ and non-zero since $w_a$ and $w_b$ are linearly independent solutions of $\tau u=0$.
With this definition the point evaluation in $c$ is given by  
 \begin{align*}
  f(c) = \spr{f}{\delta_c}, \quad f\in \Hab.
 \end{align*}
 More precisely, this follows from splitting the integrals on the right-hand side, integrating by parts twice and using the properties of the functions $w_a$, $w_b$ from Proposition~\ref{propLDdecompH1}.
 Furthermore, if the measures $\varsigma$ and $\chi$ are finite near an endpoint, say $a$, then $f(x)$ has a finite limit as $x\rightarrow a$ for each $f\in \Hab$ and
 \begin{align*}
  f(a) = \lim_{\alpha\rightarrow a} f(\alpha) = \spr{f}{\delta_a}, \quad f\in \Hab,
 \end{align*}
where the function $\delta_a$ is given by
 \begin{align}
  \delta_a(x) = - \frac{w_b(x)}{w_b^\qd(a)}, \quad x\in(a,b). 
 \end{align}
 In fact, this follows from a simple integration by parts and Proposition~\ref{propLDdecompH1}.
 If $\varsigma$ and $\chi$ are finite near the right endpoint $b$, then obviously a similar result holds for $b$.
As a consequence of this, some function $f\in \Hab$ lies in $H_0^1(a,b)$ if and only if $f$ vanishes in each endpoint near which $\varsigma$ and $\chi$ are finite.

\section{Self-adjoint Sturm--Liouville relations}\label{secSA}

In the present section we are interested in self-adjoint restrictions of the maximal relation $\Tmax$. Therefore, we will first compute its adjoint relation. 
 
\begin{theorem}\label{thmLDTmax}
 The maximal relation $\Tmax$ is closed with adjoint given by
 \begin{align}
   \Tmax^\ast = \lbrace f\in\Tmax \,|\, \forall g\in\Tmax: V(f,g)(a)=V(f,g)(b) = 0 \rbrace.
 \end{align}
\end{theorem}

\begin{proof}
 Let $\Tpre\subseteq\Tmax$ consist of all $f\in\Tmax$ such that $f_\tau\in H_c^1(a,b)$, $f$ is a scalar multiple of $w_a$ near $a$ and a scalar multiple of $w_b$ near $b$.
 Then the range of $\Tpre$ is actually equal to $H_c^1(a,b)$. Indeed, if $g\in H_c^1(a,b)$ is given, then the function
 \begin{align*}
  f(x) = W(w_b,w_a)^{-1}\left( w_b(x) \int_a^x w_a g\, d\varrho + w_a(x) \int_x^b w_b g\, d\varrho\right), \quad x\in(a,b)
 \end{align*}
 is a solution of $\tau f=g$ (see~\cite[Proposition~3.3]{measureSL}) which is a scalar multiple of $w_a$, $w_b$ near the respective endpoints and hence $g\in\ran(\Tpre)$.
 Moreover, for each $f\in\Tpre$, $g\in\Tmax$ the limits of $V(f,g)(x)$ as $x\rightarrow a$ and as $x\rightarrow b$ vanish in view of Proposition~\ref{propLDdecompH1}.
 Hence Lemma~\ref{lemLDlagrange} shows that $\Tmax\subseteq\Tpre^\ast$. Conversely, if $(f_1,f_2)\in\Tpre^\ast$, then integration by parts and Proposition~\ref{propLDdecompH1} show that
 \begin{align*}
  \spr{f_1}{g_\tau} = \spr{f_2}{g} =  \int_a^b f_2(x) g_\tau(x)^\ast d\varrho(x), \quad g\in\Tpre. 
 \end{align*}
 Now since $\ran(\Tpre)= H_c^1(a,b)$ we infer that $(f_1,f_2)\in\Tmax$ in view of Proposition~\ref{propLDweakform}.
 Thus $\Tmax$ is the adjoint of $\Tpre$ and hence closed.
 Finally we obtain
 \begin{align*}
  \Tmax^\ast = \overline{\Tpre} \subseteq \lbrace f\in\Tmax \,|\, \forall g\in\Tmax: V(f,g)(a)=V(f,g)(b) = 0 \rbrace \subseteq \Tmax^\ast,
 \end{align*}
 where we used Lemma~\ref{lemLDlagrange} and the continuity of $V(\,\cdot\,,\cdot\,)(a)$ and $V(\,\cdot\,,\cdot\,)(b)$.
\end{proof}

The adjoint of $\Tmax$ is referred to as the minimal relation $\Tmin$. This linear relation is obviously symmetric with adjoint $\Tmax$.
 Since $\Tmin$ is real with respect to the natural conjugation on $\Hab$, its deficiency indices are equal (see \cite[Theorem~4.9]{measureSL}) and at most two because there are only two linearly independent solutions of $(\tau-\I)u=0$. 
 In particular, this shows that $\Tmax$ always has self-adjoint restrictions. 
%
%
%
%
%
%
However, the actual deficiency index of $\Tmin$ depends on which cases in the following alternative (see \cite[Lemma~4]{benbro}) prevail.  At each endpoint, either 
\begin{enumerate}
\item\label{itemlcc} for every $z\in\C^\times$  all solutions of $(\tau-z)u=0$ lie in $\Hab$ near this endpoint or
\item\label{itemlpc} for every $z\in\C^\times$ there is a solution of $(\tau-z)u=0$ which does not lie in $\Hab$ near this endpoint.
\end{enumerate} 
Here and henceforth, the cross indicates that zero is removed from the respective set. 
The former case~\eqref{itemlcc} is referred to as the limit-circle (l.c.) case and the latter~\eqref{itemlpc} as the limit-point (l.p.) case.
%
%
%
 Unlike in the right-definite theory, there is a precise criterion for the l.c.\ case to prevail in terms of our measure coefficients. In fact, \cite[Theorem~3]{benbro}  shows that $\tau$ is in the l.c.\ case at an endpoint if and only if $\varsigma$, $\chi$ are finite near this endpoint and the function $\int_c^x d\varrho$, $x\in(a,b)$
 is square integrable with respect to $\varsigma$ near this endpoint for some $c\in(a,b)$. 
 Furthermore, this theorem also ensures that  all solutions of $\tau u=0$ lie in $\Hab$ near an endpoint, if $\tau$ is in the l.c.\ case there. 
 However, note that it is possible that $\tau$ is in the l.p.\ case at an endpoint although all solutions of $\tau u=0$ lie in $\Hab$ near this endpoint.
 Now along the lines of the corresponding proofs in the right-definite case \cite[Section~5]{measureSL}, one may show the following result.

\begin{theorem}\label{thmLDdefind}
The deficiency index of $\Tmin$ is given by
\begin{align}
 n(\Tmin) = \begin{cases}
              0, & \text{if }\tau\text{ is in the l.c.\ case at no endpoint}, \\
              1, & \text{if }\tau\text{ is in the l.c.\ case at precisely one endpoint}, \\
              2, & \text{if }\tau\text{ is in the l.c.\ case at both endpoints.} 
            \end{cases}
\end{align}
\end{theorem}


Furthermore, it is also possible to adapt the proof of \cite[Lemma~5.6]{measureSL}, which shows that one is able to tell from the modified Wronskian whether $\tau$ is in the l.c.\ or in the l.p.\ case.

\begin{proposition}\label{propLDlclpwronski}
 The endpoint $a$ is in the l.p.\ case if and only if $V(f,g)(a) = 0$ for every $f$, $g\in\Tmax$. If $a$ is in the l.c.\ case, then there is a $v\in\Tmax$ with 
\begin{align}\label{eqnLDlcwronsk}
 V(v,v^\ast)(a) = 0 \quad\text{and}\quad V(f,v^\ast)(a)\not=0 \quad\text{for some }f\in\Tmax.
\end{align}
Similar results hold at the endpoint $b$.
\end{proposition}

Because of the formal similarity with the right-definite theory, it is now easy to obtain a precise characterization of all self-adjoint restrictions of $\Tmax$ in terms of boundary conditions at all endpoints which are in the l.c.\ case. This can be done following literally the proofs in \cite[Section~6]{measureSL}. However, since we are only interested in separated boundary conditions we only state the following result.

\begin{theorem}\label{thmSAsep}
 Let $v_a$, $v_b\in\Tmax$ such that
 \begin{subequations}
  \begin{align}\label{eqnLDbcfunca}
   V(v_{a},v_{a}^\ast)(a) & = 0 \quad\text{and}\quad V(f,v_{a}^\ast)(a) \not=0 \quad\text{for some }f\in\Tmax, \\
  \label{eqnLDbcfuncb}
  V(v_b,v_b^\ast)(b) & = 0 \quad\text{and}\quad V(f,v_b^\ast)(b) \not=0 \quad\text{for some }f\in\Tmax,
 \end{align}\end{subequations}
 if $\tau$ is in the l.c.\ case at $a$, $b$, respectively.  Then the linear relation 
 \begin{align}\label{eqnLDselfadjsep}
  S = \lbrace f\in\Tmax \,|\, V(f,v_a^\ast)(a) = V(f,v_b^\ast)(b) = 0 \rbrace
 \end{align}
 is a self-adjoint restriction of $\Tmax$.
 \end{theorem}
 
 Note that boundary conditions at endpoints which are in the l.p.\ case are superfluous, since in this case each $f\in\Tmax$ satisfies them in view of Proposition~\ref{propLDlclpwronski}. 
 Furthermore, Theorem~\ref{thmSAsep} actually gives all possible self-adjoint restrictions of $\Tmax$ provided that $\tau$ is not at both endpoints in the l.c.\ case.


If $\tau$ is regular at an endpoint, say $a$, then the boundary condition
at this endpoint may be given in a simpler form. In fact, if some $v_a\in\Tmax$ with~\eqref{eqnLDbcfunca} is given, then it can be shown that there is some $\varphi_a\in[0,\pi)$ such that for each $f\in\Tmax$ 
\begin{align}\label{eqnLDbcreg}
 V(f,v_a^\ast)(a) = 0 \quad\Leftrightarrow\quad f_\tau(a) \cos\varphi_a - f^\qd(a) \sin\varphi_a = 0.
\end{align}
Conversely, if some $\varphi_a\in[0,\pi)$ is given, then there is a $v_a\in\Tmax$ with~\eqref{eqnLDbcfunca} such that~\eqref{eqnLDbcreg} holds for all $f\in\Tmax$. The boundary conditions corresponding to $\varphi_a=0$ are called Dirichlet boundary conditions, whereas the ones corresponding to $\varphi_a=\pi/2$ are called Neumann boundary conditions. Moreover, note that for a solution of $(\tau-z)u=0$ with $z\in\C$, the boundary condition at $a$ takes the form
\begin{align*}
 z u(a) \cos\varphi_a - u^\qd(a) \sin\varphi_a = 0.
\end{align*}



As in \cite[Corollary~8.4]{measureSL}, one may show using Proposition~\ref{propLDlclpwronski} and the Pl\"ucker identity that all non-zero eigenvalues of self-adjoint restrictions $S$ with separated boundary conditions are simple. 
However, it might happen that zero is a double eigenvalue indeed. 
 This is due to the fact that there are cases in which all solutions of $\tau u=0$ lie in $\Tmax$ and satisfy the possible boundary condition near some endpoint.
 For example, this happens for Dirichlet boundary conditions at a regular endpoint or if $\varsigma$ and $\chi$ are finite near an endpoint which is in the l.p.\ case.

\begin{theorem}\label{thmLDresolvent}
Suppose $S$ is a self-adjoint restriction of $\Tmax$ with separated boundary conditions and let $z\in\rho(S)^\times$. Furthermore, let $u_a$ and $u_b$ be non-trivial solutions of $(\tau-z)u=0$ such that
 \begin{align*}
  u_{a/b}\, \begin{cases}
         \text{satisfies the boundary condition at }a/b\text{ if }\tau\text{ is in the l.c.\ case at }a/b, \\
         \text{lies in }\Hab\text{ near }a/b\text{ if }\tau\text{ is in the l.p.\ case at }a/b.
      \end{cases}
 \end{align*}
 Then the resolvent $R_z$ is given by
 \begin{align}\label{eqnLDres}
  R_z g(x) & = \spr{g}{G_z(x,\cdot\,)^\ast}, \quad x\in(a,b),~g\in \Hab,
 \end{align}
 where
 \begin{align}
  G_z(x,y) + \frac{\delta_x(y)}{z} = \frac{1}{V(u_b,u_a)}\begin{cases}
               u_a(y) u_b(x), & \text{if }y\leq x, \\
               u_a(x) u_b(y), & \text{if }y>x. \\
             \end{cases} 
 \end{align}
 \end{theorem}

\begin{proof}
 First of all, the solutions $u_a$, $u_b$ are linearly independent, since otherwise $z$ would be an eigenvalue of $S$. 
  Now if $g\in H_c^1(a,b)$, then $f_g$ given by
 \begin{align*}
  f_g(x) = \frac{z}{V(u_b,u_a)} \left( u_b(x)\int_a^x u_a g\, d\varrho + u_a(x) \int_x^b u_b g\, d\varrho\right), \quad x\in(a,b)
 \end{align*}
 is a solution of $(\tau-z)f=g$ because of~\cite[Proposition~3.3]{measureSL}.
  Moreover, $f_g$ is a scalar multiple of $u_a$ near $a$ and a scalar multiple of $u_b$ near $b$. 
  As a consequence $f_g\in\Tmax$ satisfies the boundary conditions of $S$ and therefore $R_z g=f_g$. 
  Now an integration by parts shows that $R_z g$ is given as in~\eqref{eqnLDres}. Furthermore, by continuity this holds for all $g\in H_0^1(a,b)$.
 Hence it remains to consider $R_z w$ for $w\in\ker(\Tmax)$. In this case integration by parts yields
 \begin{align*}
  \spr{w}{G_z(x,\cdot\,)^\ast} & = \frac{1}{z} \frac{V(u_b,w)(b)}{V(u_b,u_a)(b)} u_a(x) + \frac{1}{z} \frac{V(w,u_a)(a)}{V(u_b,u_a)(a)} u_b(x) - \frac{w(x)}{z}, \quad x\in(a,b).
 \end{align*}
 Obviously, this function is a solution of $(\tau-z)f = w$, since $w$ is a solution of $\tau u=0$. 
 Moreover, if $\tau$ is in the l.p.\ case at $a$, then the second term vanishes in view of Proposition~\ref{propLDlclpwronski}.
  For the same reason the first term vanishes if $\tau$ is in the l.p.\ case at $b$ and hence this function even lies in $\Hab$.
  Using the Pl\"{u}cker identity one sees that this function also satisfies all possible boundary conditions.
\end{proof}


\section{Singular Weyl--Titchmarsh function}\label{secLDWeylTitch}

Let $S$ be some self-adjoint restriction of $\Tmax$ with separated boundary conditions as in Theorem~\ref{thmSAsep}. 
 In this section we will introduce a singular Weyl--Titchmarsh function as it has been done recently in~\cite{geszin}, \cite{kst2} and~\cite{measureSL} for the right-definite case. To this end we first need a non-trivial real analytic solution $\phi_z$, $z\in\C^\times$ of $(\tau-z)u=0$  such that $\phi_z$ lies in $S$ near $a$, i.e.\ $\phi_z$ lies in $\Hab$ near $a$ and satisfies the boundary condition at $a$ if $\tau$ is in the l.c.\ case there.

\begin{hypothesis}\label{hypLDrealentirefund}
For each $z\in\C^\times$ there is a non-trivial solution $\phi_z$ of $(\tau-z)u=0$ such that $\phi_z$ lies in $S$ near $a$ and the functions 
\begin{align}\label{eqnLDphianaly}
z\mapsto\phi_z(c) \quad\text{and}\quad z\mapsto\phi_z^\qd(c)
\end{align}
are real analytic in $\C^\times$ with at most poles at zero for each $c\in(a,b)$.
\end{hypothesis}

In order to introduce a singular Weyl--Titchmarsh function we furthermore need a second real analytic solution $\theta_z$, $z\in\C^\times$ of $(\tau-z)u=0$ with $V(\theta_z,\phi_z)=1$.

\begin{lemma}\label{lemLDtheta}
 If Hypothesis~\ref{hypLDrealentirefund} holds, then for each $z\in\C^\times$ there is a solution $\theta_z$ of $(\tau-z)u=0$ such that $V(\theta_z,\phi_z)=1$ and the functions 
 \begin{align}\label{eqnLDthetaanaly}
 z\mapsto\theta_z(c) \quad\text{and}\quad z\mapsto\theta_z^\qd(c)
 \end{align}
  are real analytic in $\C^\times$ with at most poles at zero for each $c\in(a,b)$. 
\end{lemma}

\begin{proof}
 Following literally the proof of~\cite[Lemma~2.4]{kst2} there is a real analytic solution $u_{z}$, $z\in\C^\times$ of $(\tau-z)u=0$ such that the usual Wronskian satisfies
 \begin{align*}
  W(u_z,\phi_z) = u_z(x) \phi_z^\qd(x) - u_z^\qd(x) \phi_z(x) = 1, \quad x\in(a,b),~z\in\C^\times.
 \end{align*}
 Now the solutions $\theta_z = z^{-1} u_z$, $z\in\C^\times$ have the claimed properties.
\end{proof}

Given a real analytic fundamental system $\theta_z$, $\phi_z$, $z\in\C^\times$ of $(\tau-z)u=0$ as in Hypothesis~\ref{hypLDrealentirefund} and Lemma~\ref{lemLDtheta} we may define a complex valued function $M$ on $\rho(S)^\times$ by requiring that the solutions
\begin{align}
 \psi_z = \theta_z + M(z)\phi_z, \quad z\in\rho(S)^\times
\end{align}
lie in $S$ near $b$, i.e.\ they lie in $\Tmax$ near $b$ and satisfy the boundary condition at $b$ if $\tau$ is in the l.c.\ case there. 
 Because of Theorem~\ref{thmLDdefind} and the fact that there is up to scalar multiples precisely one solution of $(\tau-z)u=0$ satisfying the boundary condition at $b$ if $\tau$ is in the l.c.\ case there, $M$ is well-defined and referred to as the singular Weyl--Titchmarsh function of $S$, associated with the fundamental system $\theta_z$, $\phi_z$, $z\in\C^\times$.

\begin{theorem}\label{thmLDmanal}
The singular Weyl--Titchmarsh function $M$ is analytic with
\begin{align}\label{eqnLDMconj}
 M(z) = M(z^\ast)^\ast, \quad z\in\rho(S)^\times.
\end{align}
\end{theorem}

\begin{proof}
 From Theorem~\ref{thmLDresolvent} 
 we get for each $c\in(a,b)$ 
 \begin{equation}\label{eqnLDresoldcdc}
  \begin{split}
   \spr{R_z \delta_c}{\delta_c} & = R_z \delta_c(c) = G_z(c,c) = \psi_z(c)\phi_z(c) - \frac{\delta_c(c)}{z} \\
           & = M(z) \phi_z(c)^2 + \theta_z(c)\phi_z(c) - \frac{w_a(c)w_b(c)}{z W(w_b,w_a)}, \quad z\in\rho(S)^\times,
  \end{split}
 \end{equation}
 which proves the claim.
\end{proof}

Similarly to the right-definite case (see e.g.\ \cite[Lemma~A.4]{kst2}, \cite[Theorem~9.4]{measureSL}), it is possible to construct a real analytic fundamental system $\theta_z$, $\phi_z$, $z\in\C^\times$ of $(\tau-z)u=0$ as in Hypothesis~\ref{hypLDrealentirefund} and Lemma~\ref{lemLDtheta}, if $\tau$ is in the l.c.\ case at $a$.

\begin{proposition}\label{propLDlcfund}
 Suppose $\tau$ is in the l.c.\ case at $a$. Then there is a real analytic fundamental system $\theta_z$, $\phi_z$, $z\in\C^\times$ of $(\tau-z)u=0$ as in Hypothesis~\ref{hypLDrealentirefund} and Lemma~\ref{lemLDtheta} such that for all $z_1$, $z_2\in\C^\times$ we additionally  have 
 \begin{align}
   V(\theta_{z_1},\phi_{z_2})(a) = 1 \quad\text{and}\quad V(\theta_{z_1},\theta_{z_2})(a) = V(\phi_{z_1},\phi_{z_2})(a) = 0.
 \end{align}
 In this case, the corresponding function $M$ is a Herglotz--Nevanlinna function.
\end{proposition}

In particular, if $\tau$ is regular at $a$ and the boundary condition there is given by
\begin{align*}
 f_\tau(a)\cos\varphi_a - f^\qd(a) \sin\varphi_a = 0, \quad f\in S
\end{align*}
for some $\varphi_a\in[0,\pi)$, then a real analytic fundamental system $\theta_z$, $\phi_z$, $z\in\C^\times$ of $(\tau-z)u=0$ is given for example by the initial conditions
\begin{align*}
 z \phi_z(a) = -\theta_z^\qd(a) = \sin\varphi_a \quad\text{and}\quad \phi_z^\qd(a) = z\theta_z(a) = \cos\varphi_a, \quad z\in\C^\times.
\end{align*}
Obviously, this fundamental system satisfies the properties in Proposition~\ref{propLDlcfund}.

As in the right-definite case (see \cite[Lemma~3.2]{geszin}, \cite[Lemma~2.2]{kst2}, \cite[Theorem~9.6]{measureSL}),  we may give a necessary and sufficient condition for Hypothesis~\ref{hypLDrealentirefund} to hold.
Therefore fix some $c\in(a,b)$ such that $\chi((a,c))\not=0$ and consider the maximal relation in $H^1(a,c)$ associated with our differential expression restricted to $(a,c)$. With $S_c$ we denote the self-adjoint restriction of this relation with  the same boundary conditions as $S$ near $a$ and Dirichlet boundary conditions at $c$.  

\begin{lemma}\label{lemLDequivhyp}
 Hypothesis~\ref{hypLDrealentirefund} holds if and only if the self-adjoint relation $S_c$ has purely discrete spectrum.
\end{lemma}

This lemma can be proved along the lines of \cite[Lemma~2.2]{kst2} and \cite[Theorem~9.6]{measureSL}.
 Moreover, if Hypothesis~\ref{hypLDrealentirefund} holds at both endpoints, then it turns out
  that the spectrum of $S$ is purely discrete.
 In particular, $S$ has purely discrete spectrum provided that $\tau$ is in the l.c.\ case at both endpoints.

\section{Spectral transformation}

In this section let $S$ again be a self-adjoint restriction of $\Tmax$ with separated boundary conditions such that Hypothesis~\ref{hypLDrealentirefund} holds.
 For the sake of simplicity we will furthermore assume that zero is not an eigenvalue of $S$. 
 This excludes for example the case of Dirichlet boundary conditions at a regular endpoint or if $\varsigma$ and $\chi$ are finite near an endpoint which is in the l.p.\ case. 

Next recall that for all functions $f$, $g\in \Hab$  there is a unique complex Borel measure $E_{f,g}$ on $\R$ such that
\begin{align}\label{eqnLDstietrans}
 \spr{R_zf}{g} = \int_\R \frac{1}{\lambda-z} dE_{f,g}(\lambda), \quad z\in\rho(S).
\end{align}
 In fact, these measures are obtained by applying a variant of the spectral theorem to the operator part
of $S$ (see e.g.~\cite[Lemma~B.4]{measureSL}). 

\begin{lemma}\label{lemLDspecmeas}
There is a unique Borel measure $\mu$ on $\R^\times$ 
such that 
\begin{align}\label{eqnLDspecmeasdens}
 E_{\delta_{\alpha},\delta_{\beta}}(B) = \int_{B^\times} \phi_\lambda(\alpha)  \phi_\lambda(\beta) d\mu(\lambda)
\end{align}
for all $\alpha$, $\beta\in(a,b)$ and each Borel set $B\subseteq\R$.
\end{lemma}

\begin{proof}
As in the proof of Theorem~\ref{thmLDmanal} one sees that for $\alpha$, $\beta\in(a,b)$
\begin{align*}
 \spr{R_z \delta_\alpha}{\delta_\beta} = M(z) \phi_z(\alpha)\phi_z(\beta) + H_{\alpha,\beta}(z), \quad z\in\rho(S)^\times,
\end{align*}
where $H_{\alpha,\beta}$ is a real analytic function on $\C^\times$. Now the claim may be deduced following the arguments in the proof of \cite[Lemma~3.3]{kst2}.
\end{proof}

Now similar to \cite[Theorem~3.4]{kst2} and \cite[Lemma~10.2]{measureSL}, we may introduce a spectral transformation for our self-adjoint linear relation $S$. 

\begin{lemma}\label{lemLDspectrans}
There is a unique bounded linear operator $\mathcal{F}:\, \Hab\rightarrow\Lrtmu$ such that for each $c\in(a,b)$ we have $\mathcal{F} \delta_c(\lambda) = \phi_\lambda(c)$ for almost all $\lambda\in\R^\times$ with respect to $\mu$. The operator $\mathcal{F}$ is a surjective partial isometry with initial subspace $\overline{\dom{S}}$. 
 Its adjoint is given by
\begin{align}
 \mathcal{F}^\ast g(x) = \int_{\R^\times} \phi_\lambda(x) g(\lambda) d\mu(\lambda), \quad x\in(a,b),~ g\in\Lrtmu,
\end{align}
its (in general multi-valued) inverse is given by
\begin{align}
 \mathcal{F}^{-1} = \lbrace (g,f)\in\Lrtmu\times\Hab \,|\, \mathcal{F}^\ast g -f\in\mul{S} \rbrace.
\end{align}
\end{lemma}

If $\tau$ is in the l.c.\ case at $b$, then the transform of a function $f\in\Hab$ is  
\begin{align*}
 \mathcal{F} f (\lambda) = \spr{\phi_\lambda}{f^\ast} = \int_a^b \phi_\lambda(x) f(x)d\chi(x) + \int_a^b \phi_\lambda^\qd(x) f^\qd(x)d\varsigma(x), \quad \lambda\in\R^\times
\end{align*}
by continuity. 
Of course, if $\tau$ is in the l.p.\ case at $b$, then this is not possible since $\phi_\lambda$ does not lie in $\Hab$ unless $\lambda$ is an eigenvalue.  However, we still have the following general result.

\begin{proposition}\label{propLDtransformint}
 If $f\in\Hab$ vanishes near $b$, then
 \begin{align}
  \mathcal{F} f(\lambda) = \int_a^b \phi_\lambda(x)f(x) d\chi(x) + \int_a^b \phi_\lambda^\qd(x)f^\qd(x) d\varsigma(x) 
 \end{align}
 for almost all $\lambda\in\R^\times$ with respect to $\mu$.
\end{proposition}

\begin{proof}
First of all, an integration by parts shows that for $\lambda\in\R^\times$ and $c\in(a,b)$
\begin{align}\tag{$*$}\label{eqnLDdeltacutoff}
 \int_a^x \phi_\lambda \delta_c\, d\chi  + \int_a^x \phi_\lambda^\qd \delta_c^\qd d\varsigma = \phi_\lambda(c) + \frac{w_a(c)w_b^\qd(x)}{W(w_b,w_a)} \phi_\lambda(x), \quad x\in(c,b).
\end{align}
Now pick some $\beta\in(a,b)$ such that $f$ vanishes on $[\beta,b)$ and consider the space $H_\beta$ of functions in $\Hab$ which are equal to a scalar multiple of $w_b$ on $[\beta,b)$. 
It is not hard to see that this space is closed and that it contains all functions $\delta_c$, $c\leq\beta$. Moreover, the linear span of these functions is even dense in $H_\beta$, i.e.\ $f$ lies in the closure of $\linspan\lbrace \delta_c \,|\, c\leq\beta\rbrace$. 
Now for each $k\in\N$ let $N(k)\in\N$ and $a_n^{k}\in\C$, $c_n^{k}\in(a,\beta)$ for $n=1,\ldots N(k)$ such that the functions
\begin{align*}
 f_k(x) = \sum_{n=1}^{N(k)} a_{n}^{k} \delta_{c_n^{k}}(x), \quad x\in(a,b),~k\in\N
\end{align*}
converge to $f$ in $\Hab$ as $k\rightarrow\infty$.
Using equation~\eqref{eqnLDdeltacutoff} we may estimate 
\begin{align*}
 & \left| \int_a^b \phi_\lambda(x)f(x)d\chi(x) + \int_a^b \phi_\lambda^\qd(x) f^\qd(x)d\varsigma(x) - \sum_{n=1}^{N(k)} a_n^{k} \phi_\lambda(c_n^{k}) \right| \\
  & \quad\quad \leq \left| \int_a^\beta \phi_\lambda (f-f_k) d\chi + \int_a^\beta \phi_\lambda^\qd (f^\qd-f_k^\qd) d\varsigma\right| + \left| \phi_\lambda(\beta) \frac{w_b^\qd(\beta)}{w_b(\beta)} f_k(\beta) \right|
\end{align*}
 for each $\lambda\in\R^\times$. 
The first term converges to zero since $f_k$ converges to $f$ in $\Hab$ as $k\rightarrow\infty$. Moreover, the second term converges to zero since $f_k(\beta)$ converges to zero as $k\rightarrow\infty$.
But this proves the claim since $\mathcal{F} f_k(\lambda)$ converges to $\mathcal{F} f(\lambda)$ for almost all $\lambda\in\R^\times$ with respect to $\mu$.
\end{proof}

If $F$ is a Borel measurable function on $\R^\times$, then we denote with $\M_F$ the maximally defined operator of multiplication with $F$ in $\Lrtmu$. 
We are now ready to state the main theorem of this section, which may be proved similarly to \cite[Theorem~10.3]{measureSL}.

\begin{theorem}\label{thmLDspecthm}
The self-adjoint relation $S$ is given by $S=\mathcal{F}^{-1} \M_{\mathrm{id}} \mathcal{F}$.
\end{theorem}

Note that all of the multi-valuedness of $S$ is only contained in the inverse of our spectral transformation.
Moreover, the self-adjoint operator part of $S$ is unitarily equivalent to the operator of multiplication $\M_{\mathrm{id}}$ in $\Lrtmu$. 
In fact, $\mathcal{F}$ is unitary as an operator from $\overline{\dom{S}}$ onto $\Lrtmu$ and maps the operator part of $S$ onto multiplication with the independent variable. 
Now the spectrum of $S$ can be read off from the boundary behavior of the singular Weyl--Titchmarsh function $M$ in the usual way (see \cite[Corollary~3.5]{kst2})
\begin{align}
 \sigma(S)^\times  = \supp(\mu) = \overline{\lbrace \lambda\in\R^\times \,|\, 0<\limsup_{\varepsilon\downarrow 0} \im\,M(\lambda+\I\varepsilon)\, \rbrace}.
\end{align}
Moreover, the point spectrum of $S$ is given by
\begin{align}
  \sigma_p(S)  = \lbrace \lambda\in\R^\times \,|\, \lim_{\varepsilon\downarrow 0} \varepsilon\, \im\, M(\lambda+\I\varepsilon)>0 \rbrace,
\end{align}
 with $\mu(\lbrace\lambda\rbrace) = \|\phi_\lambda\|^{-2}$ for all eigenvalues $\lambda\in\sigma_p(S)$.

Finally, note that the measure $\mu$ is uniquely determined by the property that the mapping $\delta_c\mapsto \phi_{\lambda}(c)$, $c\in(a,b)$ uniquely extends to a partial isometry onto $\Lrtmu$, which maps $S$ onto multiplication with the independent variable.
Because of this, the measure $\mu$ is referred to as the spectral measure of $S$ associated with the real analytic solutions $\phi_z$, $z\in\C^\times$.

\section{Associated de Branges spaces}\label{secLDdB}

As in the previous sections let $S$ be some self-adjoint restriction of $\Tmax$ (with separated boundary conditions) which does not have zero as an eigenvalue.
The aim of the present section is to describe the spaces of transforms of functions in $\Hab$ with compact support.
It will turn out that these spaces are hyperplanes in some de Branges spaces associated with our left-definite Sturm--Liouville problem, at least if we somewhat strengthen Hypothesis~\ref{hypLDrealentirefund}. In fact, in this section we will assume that for each $z\in\C$ there is a non-trivial solution $\phi_z$ of $(\tau-z)u=0$ such that $\phi_z$ lies in $S$ near $a$ and the functions
\begin{align*}
 z\mapsto\phi_z(c) \quad\text{and}\quad z\mapsto\phi_z^\qd(c)
\end{align*} 
are real entire for each $c\in(a,b)$. In particular, note that the solution $\phi_0$ is always a scalar multiple of the solution $w_a$ (due to the assumption that zero is not an eigenvalue of $S$). 
For example, if $\tau$ is regular at $a$ and the boundary condition at $a$ is given by~\eqref{eqnLDbcreg} for some $\varphi_a\in(0,\pi)$, then such a real entire solution $\phi_z$, $z\in\C$ of $(\tau-z)u=0$ is given by the initial conditions
\begin{align*}
 \phi_z(a) = \sin\varphi_a \quad \text{and}\quad \phi_z^\qd(a) = z\cos\varphi_a, \quad z\in\C.
\end{align*}
Furthermore, we will assume that the measure $\varsigma$ is absolutely continuous with respect to the Lebesgue measure. This will guarantee that our chain of de Branges spaces is continuous in some sense, which simplifies the discussion to some extend. However, we do not have to impose additional assumptions on the measures $\chi$ and $\varrho$.
 
First of all we will introduce the de Branges spaces associated with $S$ and our real entire solution $\phi_z$, $z\in\C$. 
For a brief review of de Branges' theory of Hilbert spaces of entire functions see Appendix~\ref{appLDbranges}, whereas for a detailed account we refer to de Branges' book~\cite{dBbook}. Now fix some $c\in(a,b)$ and consider the entire function
\begin{align}\label{eqnLDdeBrangesE}
E(z,c) = z\phi_z(c) + \I\phi_z^\qd(c), \quad z\in\C.
\end{align}
Then this function is a de Branges function, i.e.\ it satisfies 
\begin{align*}
|E(z,c)|>|E(z^\ast,c)|,  \quad z\in\C^+,
\end{align*}
 where $\C^+$ is the open upper complex half-plane.
Indeed, a simple calculation, using the Lagrange identity from Proposition~\ref{propLDLagrange} shows that
\begin{align*}
\frac{E(z,c)E^\#(\zeta^\ast,c)-E(\zeta^\ast,c)E^\#(z,c) }{2\I(\zeta^\ast-z)} & = \frac{\zeta^\ast \phi_\zeta(c)^\ast \phi_z^\qd(c) - z\phi_z(c) \phi_\zeta^\qd(c)^\ast }{\zeta^\ast - z} \\
 & = \int_a^c \phi_\zeta^\ast \phi_z\, d\chi + \int_a^c \phi_\zeta^{\qd\ast} \phi_z^\qd d\varsigma
\end{align*}
for each $\zeta$, $z\in\C^+$. In particular, choosing $\zeta=z$ this equality shows that our function $E(\,\cdot\,,c)$ is a de Branges function.
Hence it gives rise to a de Branges space $B(c)$ equipped with the inner product
\begin{align*}
 \dbspr{F}{G}_{B(c)} = \frac{1}{\pi}\int_\R \frac{F(\lambda)G(\lambda)^\ast}{|E(\lambda,c)|^2} d\lambda, \quad F,\,G\in B(c).  
\end{align*} 
Moreover, note that $E(\,\cdot\,,c)$ does not have any real zeros $\lambda$. Indeed, if $\lambda\not=0$ this would mean that both, $\phi_\lambda$ and its quasi-derivative vanish in $c$ and if $\lambda=0$ this would contradict the fact that $\phi_0$ is a scalar multiple of $w_a$.

The reproducing kernel $K(\,\cdot\,,\cdot\,,c)$ of the de Branges space $B(c)$ is given as in equation~\eqref{eqnLDdBrepker}. A similar calculation as above, using the Lagrange identity shows that it may be written as  
\begin{align}\label{eqnLDrepkernel}
 K(\zeta,z,c) & = \int_a^c \phi_\zeta(x)^\ast \phi_z(x)d\chi(x) + \int_a^c \phi_\zeta^\qd(x)^\ast \phi_z^\qd(x)d\varsigma(x), \quad \zeta,\, z\in\C.
\end{align}
In the following, the function $K(0,\cdot\,,c)$ will be of particular interest. An integration by parts shows that this function may as well be written as
\begin{align}\label{eqnLDdBkernelzero}
 K(0,z,c) = \phi_0^\qd(c) \phi_z(c), \quad z\in\C,
\end{align}
where the boundary term at $a$ vanishes since $\phi_0$ is a scalar multiple of $w_a$.

We want to link the de Branges space $B(c)$ to our generalized Fourier transform $\mathcal{F}$, using Proposition~\ref{propLDtransformint}.
Therefore consider the modified Sobolev space $H^1(a,c)$ and define the transform of a function $f\in H^1(a,c)$ as 
\begin{align}\label{eqnLDdBtrans}
 \hat{f}(z) = \int_a^c \phi_z(x)f(x)d\chi(x) + \int_a^c \phi_z^\qd(x)f^\qd(x)d\varsigma(x), \quad z\in\C.
\end{align}
We will now identify the de Branges space $B(c)$ with the space of transforms of functions from the subspace
\begin{align*}
 D(c) = \overline{\linspan\lbrace \phi_z|_{(a,c)} \,|\, z\in\C \rbrace}
\end{align*}
 of $H^1(a,c)$, equipped with the norm inherited from $H^1(a,c)$. 
 
\begin{theorem}\label{thmLDdBspaces}
The transformation $f\mapsto\hat{f}$ is a partial isometry from the modified Sobolev space $H^1(a,c)$ onto $B(c)$ with initial subspace $D(c)$.
\end{theorem}

\begin{proof}
 For all $\zeta\in\C$, the transform of the function $f_\zeta=\phi_\zeta^\ast|_{(a,c)}$ is given by
 \begin{align*}
  \hat{f}_\zeta(z) = \int_a^c \phi_\zeta(x)^\ast \phi_z(x) d\chi(x) + \int_a^c \phi_\zeta^\qd(x)^\ast \phi_z^\qd(x) d\varsigma(x) = K(\zeta,z,c), \quad z\in\C
 \end{align*}
 and hence lies in the de Branges space $B(c)$. Moreover, for some given $\zeta_1$, $\zeta_2\in\C$ we have
 \begin{align*}
  \spr{f_{\zeta_1}}{f_{\zeta_2}}_{H^1(a,c)} & = \int_a^c \phi_{\zeta_1}(x)^\ast \phi_{\zeta_2}(x) d\chi(x) + \int_a^c \phi_{\zeta_1}^\qd(x)^\ast \phi_{\zeta_2}^\qd(x) d\varsigma(x) \\
       & = K(\zeta_1,\zeta_2,c) = \dbspr{K(\zeta_1,\cdot\,,c)}{K(\zeta_2,\cdot\,,c)}_{B(c)} = \dbspr{\hat{f}_{\zeta_1}}{\hat{f}_{\zeta_2}}_{B(c)}.
 \end{align*}
 Now, since the functions $K(\zeta,\cdot\,,c)$, $\zeta\in\C$ are dense in $B(c)$, our transformation uniquely extends to a unitary linear map $V$ from $D(c)$ onto $B(c)$. Moreover, because the functionals $f\mapsto\hat{f}(z)$ and $f\mapsto Vf(z)$ are continuous on $D(c)$ for each fixed $z\in\C$, we conclude that $V$ is nothing but our transform restricted to $D(c)$.
 Finally, it is easily seen that transforms of functions which are orthogonal to $D(c)$ vanish identically.
\end{proof}

In the following, the closed linear subspace
\begin{align*}
B^\circ(c) = \lbrace F\in B(c) \,|\, F(0)=0 \rbrace
\end{align*}
of functions in $B(c)$ which vanish at zero will be of particular interest. 
This subspace consists precisely of all transforms of functions in $H^1(a,c)$ which vanish in $c$. In fact, an integration by parts shows that
\begin{align*}
 \hat{f}(0) =  \phi_0^\qd(c) f(c), \quad f\in H^1(a,c),
\end{align*}
where the boundary term at $a$ vanishes since $\phi_0$ is a scalar multiple of $w_a$.
Moreover, the orthogonal complement of $B^\circ(c)$ consists of all scalar multiples of the function $K(0,\cdot\,,c)$. Hence it corresponds to the one-dimensional subspace of $D(c)$ spanned by the function $\phi_0|_{(a,c)}$. 

The crucial properties of the de Branges spaces $B(c)$, $c\in(a,b)$ only hold if $c$ lies in the support 
\begin{align*}
 \supp(\varrho) = \overline{\lbrace x\in(a,b) \,|\, \forall\varepsilon>0: |\varrho|((x-\varepsilon,x+\varepsilon))>0 \rbrace}
\end{align*}
of $\varrho$. 
However, for the proof of our inverse uniqueness result a modified set $\Sigma$ instead of $\supp(\varrho)$ will be more convenient. This set $\Sigma\subseteq\supp(\varrho)\cup\lbrace a,b\rbrace$ is defined as follows.
 Take $\supp(\varrho)$ and add $a$ if $\tau$ is regular at $a$, there are no Neumann boundary conditions at $a$ and $|\varrho|$ has no mass near $a$.
  Under similar conditions one adds the endpoint $b$.
  Moreover, if $a$ has not been added, then remove the point $a_\varrho = \inf\supp(\varrho)$ unless $|\varrho|((a_\varrho,c))=0$ for some $c\in(a_\varrho,b)$. Similarly, if $b$ has not been added, then 
   remove the point $b_\varrho=\sup\supp(\varrho)$ unless $|\varrho|((c,b_\varrho))=0$ for some $c\in(a,b_\varrho)$.
 The following lemma gives a hint why this definition might be useful.

\begin{lemma}\label{lemLDdomsSigma}
The closure of the domain of $S$ is given by
\begin{align*}
 \D : = \overline{\dom{S}} = \overline{\linspan\lbrace \delta_c \,|\, c\in\Sigma \rbrace}.
\end{align*}
\end{lemma}

\begin{proof}
  The multi-valued part of $S$ is given by
  \begin{align}\tag{$*$}\label{eqnLDmulS}
   \mul{S} = \lbrace h\in\mul{\Tmax} \,|\, V((0,h),v^\ast)(a) = V((0,h),w^\ast)(b)=0 \rbrace.
  \end{align}
  Now if $c\in\Sigma\cap(a,b)$, then $\delta_c\,\bot\,\mul{S}$ since each $h\in\mul{\Tmax}$ vanishes almost everywhere with respect to $|\varrho|$.
  Moreover, if $a\in\Sigma$ then $\tau$ is regular at $a$ and there are no Neumann boundary conditions at $a$.
  Thus, each $h\in\mul{S}$ vanishes in $a$ in view of~\eqref{eqnLDmulS} and hence $\delta_a\,\bot\,\mul{S}$. Similarly one shows that $\delta_b\,\bot\,\mul{S}$ provided that $b\in\Sigma$.  
  Hence the closure of the linear span of all functions $\delta_c$, $c\in\Sigma$ is orthogonal to $\mul{S}$ and hence contained in $\D$.  In order to prove the converse let
  \begin{align*}
   h\in\overline{\linspan\lbrace \delta_c \,|\, c\in\Sigma \rbrace}^\bot.
  \end{align*}
  Since $h$ is continuous this implies that $h$ vanishes on $\supp(\varrho)$, hence $h$ lies in $\mul{\Tmax}$. 
  Now suppose that 
  \begin{align}\tag{$**$}\label{eqnLDmulScontra}
   V((0,h),v^\ast)(a) = \lim_{\alpha\rightarrow a} h(\alpha) v^\qd(\alpha)^\ast \not=0,
  \end{align}
  then $\tau$ is necessarily in the l.c.\ case at $a$. If $\varrho$ had mass near $a$, we would infer that $h(a)=0$ since $h$ vanishes on $\supp(\varrho)$. Hence $\tau$ is even regular at $a$ and~\eqref{eqnLDmulScontra} implies that there are no Neumann boundary conditions at $a$. Therefore $a$ lies in $\Sigma$ and hence $h(a)=\spr{h}{\delta_a}=0$, contradicting~\eqref{eqnLDmulScontra}. A similar argument for the right endpoint $b$ shows that $h$ lies in $\mul{S}$, which finishes the proof.
\end{proof}

Also note that functions in  $\D$ are uniquely determined by their values on $\Sigma$. 
 In fact, if $f_1$, $f_2\in\D$ such that $f_1(c)=f_2(c)$, $c\in\Sigma$, then $f_1 - f_2$ lies in the orthogonal complement of $\D$ in view of Lemma~\ref{lemLDdomsSigma} and hence $f_1=f_2$.

Now before we state our main embedding theorem, it remains to introduce the de Branges spaces $B(a)$ if $a\in\Sigma$ and $B(b)$ if $b\in\Sigma$.
 First of all if $a\in\Sigma$, then let $B(a)$ be the one-dimensional space spanned by the entire function $z\mapsto\phi_z(a)$. It does not matter which inner product this space is equipped with; each one turns $B(a)$ into a de Branges space as is easily seen from~\cite[Theorem~23]{dBbook}. In particular, note that $B^\circ(a)=\lbrace 0\rbrace$.
Finally if $b\in\Sigma$, then let $B(b)$ be the de Branges space associated with the de Branges function
\begin{align*}
 E(z,b) = z\phi_z(b) + \I\phi_z^\qd(b), \quad z\in\C.
\end{align*}
The space $B(b)$ has the same properties as the other de Branges spaces $B(c)$, $c\in(a,b)$. For example the reproducing kernel is given as in~\eqref{eqnLDrepkernel} and Theorem~\ref{thmLDdBspaces} holds with $c$ replaced by $b$.

The following result is basically a consequence of Theorem~\ref{thmLDdBspaces} and Proposition~\ref{propLDtransformint}, linking our transformation with the generalized Fourier transform $\mathcal{F}$.
In the following, $\mu$ will denote the spectral measure associated with the real analytic solutions $\phi_z$, $z\in\C^\times$ as constructed in the previous section. However, note that in the present case we may extend $\mu$ to a Borel measure on $\R$ by setting $\mu(\lbrace 0\rbrace)=0$.

\begin{theorem}\label{thmLDdBemb}
For each $c\in\Sigma$ the de Branges space $B(c)$ is a closed subspace of $\Lrmu$ with
\begin{align}\label{eqnLDdBemb}
 \spr{F}{G}_\mu = \dbspr{P^\circ F}{P^\circ G}_{B(c)} + \frac{F(0)G(0)^\ast}{|\phi_0(c)|^2} \|\delta_c\|_{\Hab}^2, \quad F,\,G\in B(c),
\end{align}
where $P^\circ$ is the orthogonal projection from $B(c)$ onto $B^\circ(c)$.
\end{theorem}

\begin{proof}
First of all note that for $z\in\C$ and $h\in\mul{S}\subseteq\mul{\Tmax}$ we have
\begin{align}\tag{$*$}\label{eqnLDpartinmul}
 \int_a^c \phi_z(x) h(x)^\ast d\chi(x) + \int_a^c \phi_z^\qd(x) h^\qd(x)^\ast d\varsigma(x) = \lim_{x\rightarrow a} \phi_z^\qd(x) h(x),
\end{align}
since $h$ vanishes almost everywhere with respect to $|\varrho|$ (in particular note that $h(c)=0$).
Moreover, the limit on the right-hand side is zero since
\begin{align*}
 \lim_{x\rightarrow a} \phi_z^\qd(x) h(x) = V\left((0,h),\phi_z\right)(a) = 0
\end{align*}
and both, $(0,h)$ and $(\phi_z,z\phi_z)$ lie in $S$ near $a$.
Now, given some arbitrary functions $f$, $g\in\linspan\lbrace\phi_z|_{(a,c)} \,|\, z\in\C\rbrace$, let
\begin{align*}
 f_0(x) = \frac{f(c)}{\phi_0(c)}\phi_0(x) \quad\text{and}\quad f_1(x) = f(x)-f_0(x), \quad x\in(a,c)
\end{align*}
and similarly for the function $g$.
The extensions $\bar{f}_1$, $\bar{g}_1$ of $f_1$, $g_1$, defined by
\begin{align*}
 \bar{f}_1(x) = \begin{cases} f_1(x), & \text{if }x\in(a,c], \\ 0, & \text{if } x\in(c,b), \end{cases}
\end{align*}
and similarly for $\bar{g}_1$, lie in $\Hab$ since $f_1(c)=g_1(c)=0$. Moreover, these extensions even lie in $\D$, because~\eqref{eqnLDpartinmul} shows that they are orthogonal to $\mul{S}$.
Now we get the identity
\begin{equation}\begin{split}\label{eqnLDallthesame}
 \spr{\hat{f}_1}{\hat{g}_1}_\mu & = \spr{\mathcal{F}\bar{f}_1}{\mathcal{F}\bar{g}_1}_\mu = \spr{\bar{f}_1}{\bar{g}_1}_{\Hab} = \spr{f_1}{g_1}_{H^1(a,c)} \\ & = \dbspr{\hat{f}_1}{\hat{g}_1}_{B(c)},
\end{split}\end{equation}
where we used Proposition~\ref{propLDtransformint}, Lemma~\ref{lemLDspectrans} and Theorem~\ref{thmLDdBspaces}.
Moreover, from~\eqref{eqnLDdBkernelzero} (also note that $\delta_c\in\D$) we get 
\begin{align*}
 \spr{\hat{f}_0}{\hat{g}_0}_\mu & = f_0(c)g_0(c)^\ast \left|\frac{\phi_0^\qd(c)}{\phi_0(c)}\right|^2 \int_\R |\phi_\lambda(c)|^2 d\mu(\lambda) = \frac{\hat{f}_0(0) \hat{g}_0(0)^\ast}{|\phi_0(c)|^2} \|\delta_c\|_{\Hab}^2.
\end{align*}
Furthermore, 
 \begin{align*}
  \spr{\hat{f}_1}{\hat{g}_0}_\mu & = g_0(c)^\ast \frac{\phi_0^\qd(c)}{\phi_0(c)} \int_\R \phi_\lambda(c) \hat{f}_1(\lambda) d\mu(\lambda) = g_0(c)^\ast \frac{\phi_0^\qd(c)}{\phi_0(c)} f_1(c) = 0,
 \end{align*}
 i.e.\ the function $\hat{g}_0$ is orthogonal to $\hat{f}_1$ not only in $B(c)$ but also in $\Lrmu$.
Using these properties, we finally obtain
 \begin{align*}
  \spr{\hat{f}}{\hat{g}}_\mu & = \spr{\hat{f}_1}{\hat{g}_1}_\mu + \spr{\hat{f}_0}{\hat{g}_0}_\mu 
                  = \spr{P^\circ \hat{f}}{P^\circ \hat{g}}_{B(c)} + \frac{\hat{f}(0) \hat{g}(0)^\ast}{|\phi_0(c)|^2} \|\delta_c\|_{\Hab}^2.
 \end{align*}
Hence~\eqref{eqnLDdBemb} holds for all $F$, $G$ in a dense subspace of $B(c)$. Now it is quite easy to see that $B(c)$ is actually continuously embedded in $\Lrmu$ and that~\eqref{eqnLDdBemb} holds for all $F$, $G\in B(c)$. Moreover, $B(c)$ is a closed subspace of $\Lrmu$ since the norms $\|\cdot\|_{B(c)}$ and $\|\cdot\|_\mu$ are equivalent on $B(c)$.
\end{proof}

In particular, note that under the assumption of Theorem~\ref{thmLDdBemb} the subspace $B^\circ(c)$ is isometrically embedded in $\Lrmu$.
Moreover, the embedding $B(c)\rightarrow\Lrmu$ preserves orthogonality and a simple calculation shows that for functions $F$ in the orthogonal complement of $B^\circ(c)$ we have 
\begin{align}
 \|F\|_{B(c)}^2 = \frac{|F(0)|^2}{\phi_0^\qd(c) \phi_0(c)} =  \left(1-\frac{w_b^\qd(c)}{w_b(c)} \frac{w_a(c)}{w_a^\qd(c)}\right) \| F\|_{\mu}^2,
\end{align}
at least if $c\not=a$. 
This difference between $B^\circ(c)$ and its orthogonal complement stems from the fact that the functions in $H^1(a,c)$ corresponding to $B^\circ(c)$ are isometrically embedded in $\Hab$, whereas the functions corresponding to its orthogonal complement are not.

The following results contain further properties of our de Branges spaces which are needed for the inverse uniqueness theorem in the next section. First of all, we will show that they are totally ordered and strictly increasing.

\begin{proposition}\label{propLDdBincl}
If $c_1$, $c_2\in\Sigma$ with $c_1<c_2$, then 
\begin{align*}
B(c_1)\subsetneq B(c_2).
\end{align*} 
 Moreover, if $|\varrho|((c_1,c_2))=0$ then $B(c_1)$ has codimension one in $B(c_2)$.
\end{proposition}

\begin{proof}
 If $\delta\in H^1(a,c_2)$ is such that
 \begin{align*}
  \spr{f}{\delta}_{H^1(a,c_2)} = f(c_1), \quad f\in H^1(a,c_2),
 \end{align*}
 then the modified Sobolev space $H^1(a,c_2)$ may be decomposed into
 \begin{align*}
  H^1(a,c_2) = H_{-}^1(a,c_2)\oplus \linspan\lbrace\delta\rbrace \oplus H_{+}^1(a,c_2).
 \end{align*}
 Here $H_-^1(a,c_2)$ is the subspace of functions in $H^1(a,c_2)$ vanishing on $(c_1,c_2)$ and $H_+^1(a,c_2)$ is the subspace of functions in $H^1(a,c_2)$ vanishing on $(a,c_1)$.
 Now the transforms of functions in $H_-^1(a,c_2)$ are precisely the transforms of functions in $H^1(a,c_1)$ which vanish in $c_1$, i.e.\ $B^\circ(c_1)$.
  The transform of the subspace $\linspan\lbrace\delta\rbrace$ is precisely the orthogonal complement of $B^\circ(c_1)$. 
 Hence one sees that $B(c_1)$ is contained in $B(c_2)$.  In order to prove that $B(c_2)$ is larger indeed, suppose that the function $z\mapsto\phi_z(c_2)$ belongs to $B(c_1)$.
  Since this function is orthogonal to $B^\circ(c_2)$ it is also orthogonal to $B^\circ(c_1)$ by Theorem~\ref{thmLDdBemb}.  Thus we infer that the functions $z\mapsto\phi_z(c_1)$ and $z\mapsto\phi_z(c_2)$ are linearly dependent. Now from Lemma~\ref{lemLDspectrans} (hereby also note that $\delta_{c_1}$ and $\delta_{c_2}$ lie in $\D$) one sees that $\delta_{c_1}$ and $\delta_{c_2}$ are also linearly dependent, which gives a contradiction. 
  
 It remains to prove that the space of transforms of functions in $H_+^1(a,c_2)$ is at most one-dimensional provided that $|\varrho|((c_1,c_2))=0$.
  Indeed, for each function $f\in H_+^1(a,c_2)$ an integration by parts shows that
 \begin{align*}
  \hat{f}(z) = \phi_z^\qd(c_2)f(c_2) - \phi_z^\qd(c_1) f(c_1) + z \int_{c_1}^{c_2} \phi_z f\, d\varrho = \phi_z^\qd(c_2) f(c_2), \quad z\in\C,
 \end{align*}
 since $f$ vanishes on $(a,c_1]$ and $|\varrho|((c_1,c_2))=0$. 
\end{proof}

The following result shows that our de Branges spaces are continuous in some sense.
This is due to the assumption that the measure $\varsigma$ is absolutely continuous with respect to the Lebesgue measure. Otherwise, there would be jumps of dimension one in points where $\varsigma$ has mass.

\begin{proposition}\label{propLDcontinuous}
If $c$, $\alpha_n$, $\beta_n\in\supp(\varrho)$, $n\in\N$ are such that $\alpha_n\uparrow c$ and $\beta_n\downarrow c$ as $n\rightarrow\infty$, then 
\begin{align}\label{eqnLDdBcont}
 \overline{\bigcup_{n\in\N} B(\alpha_n)} = B(c) = \bigcap_{n\in\N} B(\beta_n),
\end{align}
where the closure is taken in $\Lrmu$.
\end{proposition}

\begin{proof}
From Proposition~\ref{propLDdBincl} it is clear that
\begin{align*}
 \overline{\bigcup_{n\in\N} B(\alpha_n)} \subseteq B(c) \subseteq \bigcap_{n\in\N} B(\beta_n).
\end{align*}
If $F\in B^\circ(c)$, then there is an $f\in H^1(a,c)$ with $f(c)=0$ such that $\hat{f}=F$. Now choose a sequence $f_k\in H^1(a,c)$, $k\in\N$ of functions which vanish near $c$, such that $f_k\rightarrow f$ as $k\rightarrow\infty$.
By our assumptions the transform of each of these functions lies in $B(\alpha_n)$, provided that $n\in\N$ is large enough, i.e.
\begin{align*}
 \hat{f}_k \in \bigcup_{n\in\N} B(\alpha_n), \quad k\in\N.
\end{align*} 
Consequently the transform of $f$ lies in the closure of this union. Moreover, for each $n\in\N$ the entire function $z\mapsto\phi_z(\alpha_n)$ lies in $B(\alpha_n)$. Now since $\delta_{\alpha_n}\rightarrow\delta_c$ in $\Hab$, Lemma~\ref{lemLDspectrans} shows that the entire function $z\mapsto\phi_z(c)$ lies in the closure of our union which proves the first equality in~\eqref{eqnLDdBcont}. 

Next, if $F\in B(\beta_n)$ for each $n\in\N$, then there are $f_n\in D(\beta_n)$ such that
\begin{align*}
 F(z) = \int_a^{\beta_n} \phi_z(x) f_n(x) d\chi(x) + \int_a^{\beta_n} \phi_z^\qd(x) f_n^\qd(x) d\varsigma(x), \quad z\in\C,~n\in\N.
\end{align*}
Moreover, from Theorem~\ref{thmLDdBspaces} and Theorem~\ref{thmLDdBemb} we infer
\begin{align*}
 \|f_n\|_{H^1(a,\beta_n)}^2 = \|F\|_{B(\beta_n)}^2 \leq \left(1+\left|\frac{w_b^\qd(\beta_n)}{w_b(\beta_n)} \frac{w_a(\beta_n)}{w_a^\qd(\beta_n)}\right|\right) \|F\|_\mu^2, \quad n\in\N,
\end{align*}
where the coefficient on the right-hand side is bounded uniformly for all $n\in\N$ by the properties of the solutions $w_a$ and $w_b$.
Hence there is some subsequence of $f_n|_{(a,c)}$, $n\in\N$ converging weakly in $H^1(a,c)$ to say $f$. Now this yields for all $z\in\C$ 
\begin{align*}
 F(z) = \hat{f}(z) + \chi(\lbrace c\rbrace) \phi_z(c) f(c) + \lim_{n\rightarrow\infty} \int_{(c,\beta_n)} \phi_z f_n\, d\chi +\int_{(c,\beta_n)} \phi_z^\qd f_n^\qd d\varsigma,
\end{align*}
where the limit is actually zero. In fact, for each $z\in\C$ and $n\in\N$ we have
\begin{align*}
  \left|\int_{(c,\beta_n)} \phi_z f_n d\chi +\int_{(c,\beta_n)} \phi_z^\qd f_n^\qd d\varsigma\right|  
    \leq C_z^2 \|f_n\|_{H^1(a,\beta_n)} (\chi((c,\beta_n)) + \varsigma((c,\beta_n))),
\end{align*}
where $C_z\in\R$ is such that the moduli of $\phi_z$ and $\phi_z^\qd$ on $(c,\beta_1)$ are bounded by $C_z$. 
But this shows that $F$ actually is the transform of a function in $\Hab$ and hence lies in $B(c)$ which finishes the proof.
\end{proof}

Finally we will prove that our de Branges spaces decrease to zero near $a$ and fill the whole space $\Lrmu$ near $b$. 

\begin{proposition}\label{propLDemptydense}
The de Branges spaces $B(c)$, $c\in\Sigma$ satisfy
\begin{align}
 \bigcap_{c\,\in\,\Sigma} B^\circ(c) = \lbrace 0\rbrace \quad\text{and}\quad \overline{\bigcup_{c\,\in\,\Sigma} B(c)} = \Lrmu. 
\end{align}
\end{proposition}

\begin{proof}
 First suppose that $\supp(\varrho)\cap(a,c)\not=\emptyset$ for each $c\in(a,b)$ and pick some $F\in\bigcap_{c\in\Sigma} B^\circ(c)$. Then for each $\zeta\in\C$ we have
 \begin{align*}
  |F(\zeta)| & \leq \dbspr{F}{K(\zeta,\cdot\,,c)}_{B(c)} 
             \leq \|F\|_{B(c)} \dbspr{K(\zeta,\cdot\,,c)}{K(\zeta,\cdot\,,c)}_{B(c)} \\
            & \leq    \|F\|_{\mu} K(\zeta,\zeta,c)
\end{align*} 
for each $c\in\supp(\varrho)$.
Now from~\eqref{eqnLDrepkernel} we infer that $K(\zeta,\zeta,c)\rightarrow 0$ as $c\rightarrow a$ and hence that $F=0$.
 Otherwise, if $\alpha_\varrho=\inf\supp(\varrho)>a$, then the subspace  
 \begin{align*}
 D^\circ(\alpha_\varrho) = \lbrace f\in D(\alpha_\varrho) \,|\, f(\alpha_\varrho)=0 \rbrace,
 \end{align*} 
 corresponding to $B^\circ(\alpha_\varrho)$, is at most one-dimensional. 
 In fact, this is because each function $\phi_z|_{(a,\alpha_\varrho)}$, $z\in\C$ is a solution of $\tau u=0$ on $(a,\alpha_\varrho)$ in this case.
 Consequently, the functions in $D^\circ(\alpha_\varrho)$ are also solutions of $\tau u=0$ on $(a,\alpha_\varrho)$.
 Moreover, if $\varsigma+\chi$ is infinite near $a$, then each $f\in D^\circ(\alpha_\varrho)$ is a scalar multiple of $w_a$ on $(a,\alpha_\varrho)$ with $f(\alpha_\varrho)=0$ and hence vanishes identically.
 Also if $\varsigma+\chi$ is finite near $a$ and there are Neumann boundary conditions at $a$, one sees that $f$ is a scalar multiple of $w_a$ and hence identically zero.
 We conclude that the first equality in the claim holds in these cases.
 Finally, if $\varsigma+\chi$ is finite near $a$ and there are no Neumann boundary conditions at $a$, then $a\in\Sigma$ and hence clearly $B^\circ(a)=\lbrace 0\rbrace$.
 For the second equality note that the linear span of functions $z\mapsto\phi_z(c)$, $c\in\Sigma$ is dense in $\Lrmu$ in view of Lemma~\ref{lemLDdomsSigma} and Lemma~\ref{lemLDspectrans}.
 \end{proof}

\begin{remark} 
 At this point let us mention that a real entire solution $\phi_z$, $z\in\C$ as in this section is not unique. In fact, any other such solution is given by 
 \begin{align*}
   \tilde{\phi}_z = \E^{g(z)} \phi_z, \quad z\in\C
 \end{align*}
 for some real entire function $g$. 
%
%
 The corresponding spectral measures are related by
\begin{align*}
 \tilde{\mu}(B) = \int_B \E^{-2g(\lambda)} d\mu(\lambda)
\end{align*}
for each Borel set $B\subseteq\R$. 
 In particular, the measures are mutually absolutely continuous and the associated spectral transforms just differ by a simple rescaling with a positive function.
 Moreover, from Theorem~\ref{thmLDdBspaces} it is easily seen that for each $c\in(a,b)$, multiplication with the entire function $\E^{-g}$ maps $B(c)$ isometrically onto the corresponding de Branges space $\tilde{B}(c)$.
\end{remark}

\section{Inverse uniqueness results}\label{secLDuniqeness}

 The present section is devoted to our inverse uniqueness result.
 We will prove that the spectral measure determines a left-definite Sturm--Liouville operator up to some Liouville transformation (see e.g.\ \cite{bebrwe} or~\cite{ben1} for the right-definite case).
 Therefore let $S_1$ and $S_2$ be two self-adjoint left-definite Sturm--Liouville relations (with separated boundary conditions), both satisfying the assumptions made in the previous section, i.e.\ zero is not an eigenvalue of $S_1$ and $S_2$ and there are real entire solutions satisfying the boundary condition at the left endpoint. Moreover, again we assume that the measures $\varsigma_1$ and $\varsigma_2$ are absolutely continuous with respect to the Lebesgue measure.
 All remaining quantities corresponding to $S_1$ respectively $S_2$ are denoted with an additional subscript.

We will first state a part of the proof of our inverse uniqueness result as a separate lemma.
Note that the equality in the claim of this lemma has to be read as sets of entire functions and not as de Branges spaces.
In general the norms of these spaces will differ from each other.

\begin{lemma}\label{lemLDeta}
Suppose that the function 
\begin{align}\label{eqnLDcondboundtype}
 \frac{E_1(z,x_1)}{E_2(z,x_2)}, \quad z\in\C^+
\end{align}
is of bounded type for some $x_1\in\Sigma_1$ and $x_2\in\Sigma_2$. 
If $\mu_1 = \mu_2$, then there is an increasing continuous bijection $\eta$ from $\Sigma_1$ onto $\Sigma_2$ such that 
\begin{align*}
B_1(x_1)=B_2(\eta(x_1)), \quad x_1\in\Sigma_1.
\end{align*}
\end{lemma}

\begin{proof}
 First of all note that by the definition of de Branges spaces and Proposition~\ref{propLDdBincl} the function in~\eqref{eqnLDcondboundtype} is of bounded type for all $x_1\in\Sigma_1$ and $x_2\in\Sigma_2$. 
 We will first consider the case when $\Sigma_1$ consists of finitely many (strictly increasing) points $x_{1,n}$, $n=1,\ldots,N$ separately. 
 In this case $\mu_1=\mu_2$ is supported on $N$ points, since $\mathcal{F}_1$ is a unitary map from $\D_1$ onto $L^2(\R;\mu_1)$. 
 Hence, $\Sigma_2$ also consists of finitely many (strictly increasing) points $x_{2,n}$, $n=1,\ldots,N$. 
 Now let $\eta$ be the unique strictly increasing bijection from $\Sigma_1$ onto $\Sigma_2$, i.e.\ $\eta(x_{1,n})=x_{2,n}$, $n=1,\ldots,N$.
 Using the properties of our de Branges spaces it is quite simple to see that 
 \begin{align*}
  \dim B_1(x_{1,n}) = \dim B_2(x_{2,n}) = n, \quad n=1,\ldots,N,
 \end{align*}
 and therefore the claim follows from Theorem~\ref{thmLDdBemb} and Theorem~\ref{thmLDdBordpitt}.
 
 Now suppose that $\Sigma_1$ consists of infinitely many points and fix some arbitrary $x_1\in\Sigma_1\backslash\lbrace \inf\Sigma_1,\sup\Sigma_1\rbrace$. Then from Theorem~\ref{thmLDdBemb} and Theorem~\ref{thmLDdBordpitt} we infer that for each $x_2\in\Sigma_2$ either $B_1(x_1)\subseteq B_2(x_2)$ or $B_1(x_1)\supseteq B_2(x_2)$ and hence also $B_1^\circ(x_1)\subseteq B_2^\circ(x_2)$ or $B_1^\circ(x_1)\supseteq B_2(x_2)$.
  In order to define $\eta(x_1)\in(a_2,b_2)$ we are first going to show that both of the sets
\begin{align*}
 J_- & = \lbrace x_2\in\Sigma_2 \,|\, B_2(x_2) \subsetneq B_1(x_1) \rbrace, \\
 J_+ & = \lbrace x_2\in\Sigma_2 \,|\, B_1(x_1) \subsetneq B_2(x_2) \rbrace,
\end{align*} 
 are non-empty. Indeed, if $J_-$ was empty, then $B_1^\circ(x_1)\subseteq B_2^\circ(x_2)$ for each $x_2\in\Sigma_2$ and hence
 \begin{align*}
  B_1^\circ(x_1) \subseteq \bigcap_{x_2\in\Sigma_2} B_2^\circ(x_2) = \lbrace 0\rbrace,
 \end{align*}
 in view of Proposition~\ref{propLDemptydense}. 
 Thus we obtained the contradiction $x_1=\inf\Sigma_1$, since otherwise there would be some $\tilde{x}_1\in\Sigma_1$ with $\tilde{x}_1<x_1$ such that $B_1(\tilde{x}_1)\subsetneq B_1(x_1)$.
 Furthermore, if $J_+$ was empty, then $B_2(x_2) \subseteq B_1(x_1)$ for each $x_2\in\Sigma_2$ and hence
 \begin{align*}
  L^2(\R;\mu_1) = \overline{\bigcup_{x_2\in\Sigma_2} B_2(x_2)} \subseteq B_1(x_1) \subseteq L^2(\R;\mu_1).
 \end{align*}
 But from this we infer the contradiction $x_1=\sup\Sigma_1$, since otherwise there would be an $\tilde{x}_1\in\Sigma_1$ with $\tilde{x}_1>x_1$ such that $B_1(x_1) \subsetneq B_1(\tilde{x}_1) \subseteq L^2(\R;\mu_1)$.
  Hence we showed that $J_-$ and $J_+$ are non-empty. 
  Now, if $J_-=\lbrace a_2\rbrace$ then the space $B_2(\alpha_{\varrho_2})$ is two-dimensional and $\alpha_{\varrho_2}$ does not lie in $J_+$ since otherwise
 \begin{align*}
  B_2(a_2) \subsetneq B_1(x_1) \subsetneq B_2(\alpha_{\varrho_2}).
 \end{align*}
 Thus in this case we may set $\eta(x_1)=\alpha_{\varrho_2}$ and obtain $B_1(x_1)=B_2(\eta(x_1))$. 
 Furthermore, if $J_+=\lbrace b_2\rbrace$ then the space $B_2(\beta_{\varrho_2})$ has codimension one in $L^2(\R;\mu)$ and $\beta_{\varrho_2}$ does not lie in $J_-$ since otherwise
 \begin{align*}
  B_2(\beta_{\varrho_2}) \subsetneq B_1(x_1) \subsetneq B_2(b_2).
 \end{align*}
 Again, we may define $\eta(x_1)=\beta_{\varrho_2}$ and get $B_1(x_1)=B_2(\eta(x_1))$. 
 Now in the remaining cases $J_-$ is bounded from above in $(a_2,b_2)$ with supremum
  \begin{align*}
   \eta_-(x_1) = \sup J_- \in(a_2,b_2),
  \end{align*}
 and $J_+$ is bounded from below in $(a_2,b_2)$ with infimum
  \begin{align*}
   \eta_+(x_1) = \inf J_+ \in(a_2,b_2).
  \end{align*}
 Moreover, we have $\eta_\pm(x_1)\in\supp(\varrho_2)$ since $J_\pm\backslash\lbrace a_2,b_2\rbrace$ is contained in $\supp(\varrho_2)$.
 Now Proposition~\ref{propLDcontinuous} shows that
 \begin{align*}
  B_2(\eta_-(x_1)) \subseteq B_1(x_1) \subseteq B_2(\eta_+(x_1)).
 \end{align*}
 If $B_1(x_1)=B_2(\eta_-(x_1))$, set $\eta(x_1)=\eta_-(x_1)$ and if $B_1(x_1)=B_2(\eta_+(x_1))$, set $\eta(x_1)=\eta_+(x_1)$ to obtain $B_1(x_1)=B_2(\eta(x_1))$. Otherwise we have
  \begin{align*}
  B_2(\eta_-(x_1)) \subsetneq B_1(x_1) \subsetneq B_2(\eta_+(x_1)),
 \end{align*}
 and hence $\supp(\varrho_2)\cap(\eta_-(x_1),\eta_+(x_1))\not=\emptyset$ in view of Proposition~\ref{propLDdBincl}. Now we may choose $\eta(x_1)$ in this intersection and get $B_1(x_1)=B_2(\eta(x_1))$ since $\eta(x_1)$ neither lies in $J_-$ nor in $J_+$.

 Up to now we constructed a function $\eta: \Sigma_1\backslash\lbrace\inf\Sigma_1,\sup\Sigma_1\rbrace \rightarrow \Sigma_2$ such that $B_1(x_1) = B_2(\eta(x_1))$ for each $x_1\in\Sigma_1\backslash\lbrace\inf\Sigma_1,\sup\Sigma_1\rbrace$.
  Now if $\inf\Sigma_1$ lies in $\Sigma_1$ and we set $x_1=\inf\Sigma_1\backslash\lbrace\inf\Sigma_1\rbrace$, then $B_1(x_1) = B_2(\eta(x_1))$ is two- dimensional and from Proposition~\ref{propLDemptydense} we infer that there is an $x_2\in\Sigma_2$ with
  \begin{align*}
   \lbrace 0\rbrace \subsetneq B_2(x_2) \subsetneq B_2(\eta(x_1)) = B_1(x_1).
  \end{align*}
  Hence we may set $\eta(\inf\Sigma_1)=x_2$ and obtain $B_1(\inf\Sigma_1) = B_2(\eta(\inf\Sigma_1))$.
  Similarly, if $\sup\Sigma_1$ lies in $\Sigma_1$ and we set $x_1=\sup\Sigma_1\backslash\lbrace\sup\Sigma_1\rbrace$, then the space  $B_1(x_1) = B_2(\eta(x_1))$ has codimension one in $B_1(\sup\Sigma_1) = L^2(\R;\mu_1)$. But because of Proposition~\ref{propLDemptydense} there is an $x_2\in\Sigma_2$ such that 
  \begin{align*}
   B_2(\eta(\beta_{\varrho_1})) \subsetneq B_2(x_2) \subseteq  L^2(\R;\mu_1).
  \end{align*}
  Again, we may define $\eta(\sup\Sigma_1)=x_2$ and get $B_1(\sup\Sigma_1) = B_2(\eta(\sup\Sigma_1))$. 
  Thus, we extended our function $\eta$ to all of $\Sigma_1$ and are left to prove the remaining claimed properties.
  
  The fact that $\eta$ is increasing is a simple consequence of Proposition~\ref{propLDdBincl}. Now if $x_2\in\Sigma_2$, then the first part of the proof with the roles of $\Sigma_1$ and $\Sigma_2$ reversed shows that there is an $x_1\in\Sigma_1$ with $B_1(x_1) = B_2(x_2) = B_1(\eta(x_1))$. 
   In view of Proposition~\ref{propLDdBincl} this yields $\eta(x_1)=x_2$ and hence $\eta$ is a bijection. Finally, continuity follows from Proposition~\ref{propLDcontinuous}. Indeed, if $c$, $c_n\in\Sigma_1$, $n\in\N$ such that $c_n\uparrow c$ as $n\rightarrow\infty$, then
 \begin{align*}
  B_2\left(\lim_{n\rightarrow\infty} \eta(c_n)\right) = \overline{\bigcup_{n\in\N} B_2(\eta(c_n))} = \overline{\bigcup_{n\in\N} B_1(c_n)} = B_1(c) = B_2(\eta(c))
 \end{align*}
 and hence $\eta(c_n)\rightarrow \eta(c)$ as $n\rightarrow\infty$. Similarly, if $c_n\downarrow c$ as $n\rightarrow\infty$, then
  \begin{align*}
  B_2\left(\lim_{n\rightarrow\infty} \eta(c_n)\right) = \bigcap_{n\in\N} B_2(\eta(c_n)) = \bigcap_{n\in\N} B_1(c_n) = B_1(c) = B_2(\eta(c))
 \end{align*}
 and hence again $\eta(c_n)\rightarrow \eta(c)$ as $n\rightarrow\infty$.
\end{proof}

Note that the condition that the function in~\eqref{eqnLDcondboundtype} is of bounded type is actually equivalent to the function
\begin{align*}
 \frac{\phi_{1,z}(x_1)}{\phi_{2,z}(x_2)}, \quad z\in\C^+
\end{align*} 
being of bounded type for some $x_1\in\Sigma_1$ and $x_2\in\Sigma_2$. 
Unfortunately, these conditions are somewhat inconvenient in view of applications. 
However, note that this assumption is for example fulfilled if for some $x_1\in\Sigma_1$ and $x_2\in\Sigma_2$ the entire functions $z\mapsto\phi_{1,z}(x_1)$ and  $z\mapsto\phi_{2,z}(x_2)$ are of finite exponential type such that the logarithmic integrals
\begin{align*}
 \int_\R \frac{\ln^+|\phi_{j,\lambda}(x_j)|}{1+\lambda^2} d\lambda < \infty, \quad j=1,2
\end{align*}
are finite. Here $\ln^+$ is the positive part of the natural logarithm.
Indeed, a theorem of Krein~\cite[Theorem~6.17]{lev}, \cite[Section~16.1]{rosrov} states that in this case the functions $z\mapsto\phi_{j,z}(x_j)$, $j=1,2$ (and hence also their quotient) are of bounded type in the upper and in the lower complex half-plane. Moreover, note that the conclusion of Lemma~\ref{lemLDeta} is also true if for some (and hence all) $x_1\in(a_1,b_1)$ and $x_2\in(a_2,b_2)$ the functions $E_1(\,\cdot\,,x_1)$, $E_2(\,\cdot\,,x_2)$ are of exponential type zero, i.e. 
\begin{align*}
 \ln^+ |E_j(z,x_j)| = \oo(|z|), \quad j=1,2
\end{align*}
as $|z|\rightarrow\infty$ in $\C$. 
 The proof therefore is literally the same, except that one has to apply Theorem~\ref{thmLDdBordkotani} instead of Theorem~\ref{thmLDdBordpitt}.

With all the work done in Lemma~\ref{lemLDeta} it is now quite simple to show that the spectral measure determines our self-adjoint Sturm--Liouville relation up to a Liouville transform. Here, a Liouville transform $\mathcal{L}$ is a unitary map from $\D_2$ onto $\D_1$ given by
\begin{align}
 \mathcal{L} f_2(x_1) = \kappa(x_1) f_2(\eta(x_1)), \quad x_1\in\Sigma_1,~ f_2\in\D_2,
\end{align}
where $\eta$ is an increasing continuous bijection from $\Sigma_1$ onto $\Sigma_2$ and $\kappa$ is a non-vanishing real function on $\Sigma_1$.
We say that the Liouville transform $\mathcal{L}$ maps $S_1$ onto $S_2$ if
\begin{align*}
  S_2 = \mathcal{L}^\ast S_1 \mathcal{L},
\end{align*}
where $\mathcal{L}^\ast$ is the adjoint of $\mathcal{L}$ regarded as a linear relation in $H^1(a_2,b_2)\times H^1(a_1,b_1)$. 
Note that in this case the operator parts of $S_1$, $S_2$ are unitarily equivalent in view of this Liouville transform $\mathcal{L}$.

\begin{theorem}\label{thmLDuniq}
Suppose that the function
\begin{align*}
  \frac{E_1(z,x_1)}{E_2(z,x_2)}, \quad z\in\C^+
\end{align*}
is of bounded type for some $x_1\in\Sigma_1$ and $x_2\in\Sigma_2$. 
If $\mu_1 = \mu_2$, then there is a Liouville transform $\mathcal{L}$ mapping $S_1$ onto $S_2$.
\end{theorem}

\begin{proof}
By Lemma~\ref{lemLDeta} there is an increasing continuous bijection $\eta$ from $\Sigma_1$ onto $\Sigma_2$ such that 
$B_1(x_1)=B_2(\eta(x_1))$ and hence also $B_1^\circ(x_1)=B_2^\circ(\eta(x_1))$ for each $x_1\in\Sigma_1$.
According to Theorem~\ref{thmLDdBemb}, for each fixed $x_1\in\Sigma_1$ the entire functions 
\begin{align*}
z\mapsto\phi_{1,z}(x_1) \quad\text{and}\quad z\mapsto\phi_{2,z}(\eta(x_1))
\end{align*}
 are orthogonal to $B_1^\circ(x_1)=B_2^\circ(\eta(x_1))$ in $L^2(\R;\mu_1)$. From this we infer that
\begin{align}\label{eqnLDphiLT}
  \phi_{1,z}(x_1) = \kappa(x_1) \phi_{2,z}(\eta(x_1)), \quad z\in\C
\end{align}
for some $\kappa(x_1)\in\R^\times$ and hence also
\begin{align}\tag{$*$}\label{eqnLDF1F2}
 \mathcal{F}_1 \delta_{1,x_1} = \kappa(x_1) \mathcal{F}_2 \delta_{2,\eta(x_1)}.
\end{align}
Now the linear relation
\begin{align*}
 \mathcal{L} = \mathcal{F}_1^\ast \mathcal{F}_{2}|_{\D_2}
\end{align*}
is a unitary mapping from $\D_2$ onto $\D_1$ by Lemma~\ref{lemLDspectrans} and  moreover, equation~\eqref{eqnLDF1F2} shows that
\begin{align*}
(\delta_{1,x_1}, \kappa(x_1) \delta_{2,\eta(x_1)}) \in \mathcal{L}^\ast = \mathcal{F}_{2}^{-1} \mathcal{F}_1, \quad x_1\in\Sigma_1.
\end{align*}
From this one sees that the transform of some function $f_2\in\D_2$ is given by 
\begin{align*}
 \mathcal{L} f_2(x_1) & = \spr{\mathcal{L} f_2}{\delta_{1,x_1}}_{H^1(a_1,b_1)} = \kappa(x_1) \spr{f_2}{\delta_{2,\eta(x_1)}}_{H^1(a_2,b_2)} \\
                      &  = \kappa(x_1) f_2(\eta(x_1))
\end{align*}
at each point $x_1\in\Sigma_1$. Finally, we conclude that 
\begin{align*}
 S_2 & = \mathcal{F}_{2}^{-1} \M_{\mathrm{id}} \mathcal{F}_{2}|_{\D_2} = \mathcal{F}_{2}^{-1} \mathcal{F}_{1} \mathcal{F}_{1}^{-1} \M_{\mathrm{id}} \mathcal{F}_{1} \mathcal{F}_{1}^\ast \mathcal{F}_{2}|_{\D_2} =  \mathcal{L}^\ast \mathcal{F}_{1}^{-1} \M_{\mathrm{id}} \mathcal{F}_{1} \mathcal{L} \\ & = \mathcal{L}^\ast S_1 \mathcal{L},
\end{align*}
from Theorem~\ref{thmLDspecthm}. 
\end{proof}

We will now show to which extend the spectral measure determines the coefficients.  
For the proof we need a result on the high energy asymptotics of solutions of our differential equation (see e.g.\ \cite[Section~6]{ben2}). Henceforth we will denote with $r_j$, $j=1,2$ the densities of the absolute continuous parts of $\varrho_j$ with respect to the Lebesgue measure and with $p_j^{-1}$, $j=1,2$ the densities of $\varsigma_j$ with respect to the Lebesgue measure.

\begin{lemma}\label{lemLDasympt}
 For each $j=1,2$ and all points $x_j$, $\tilde{x}_j\in(a_j,b_j)$ we have the asymptotics 
 \begin{align*}
  \sqrt{\frac{2}{y}} \ln \frac{|\phi_{j,\I y}(x_j)|}{|\phi_{j,\I y}(\tilde{x}_j)|} \rightarrow  \int_{\tilde{x}_j}^{x_j} \sqrt{\frac{|r_j(x)|}{p_j(x)}} dx,
 \end{align*}
 as $y\rightarrow\infty$ in $\R^+$.
\end{lemma}

\begin{proof}
 By our assumptions, the Lebesgue decomposition of the measure $\varrho_j$ with respect to $\varsigma_j$ is given by
 \begin{align*}
  \varrho_j= r_j p_j \varsigma_j + \varrho_{j,s},
 \end{align*}
  where $\varrho_{j,s}$ is the singular part of $\varrho_j$ with respect to the Lebesgue measure.
  Now the results in~\cite[Section~6]{ben2} show that (the square root is the principal one with branch cut along the negative real axis)
 \begin{align*}
  \ln\frac{|\phi_{j,\I y}(x_j)|}{|\phi_{j,\I y}(\tilde{x}_j)|} & = \re\left(\int_{\tilde{x}_j}^{x_j} \sqrt{-\I y r_j(x) p_j(x)}d\varsigma_j(x) + \oo(\sqrt{y})\right) \\
   & = \sqrt{\frac{y}{2}} \int_{\tilde{x}_j}^{x_j} \sqrt{\frac{|r_j(x)|}{p_j(x)}}dx + \oo(\sqrt{y}),
 \end{align*}
 as $y\rightarrow\infty$ in $\R^+$, which yields the claim. 
\end{proof}

We are now able to establish a relation between the measure coefficients. However, this is only possible on sets where the support of the weight measure has enough density. Otherwise there would be to much freedom for the remaining coefficients. 
 
\begin{corollary}\label{corLDuniqcoef}
Let $\alpha_1$, $\beta_1\in(a_1,b_1)$ with $\alpha_1<\beta_1$ such that $r_1\not=0$ almost everywhere on $(\alpha_1,\beta_1)$ and $r_2\not=0$ almost everywhere on $(\eta(\alpha_1),\eta(\beta_1))$ with respect to the Lebesgue measure.
If the function 
\begin{align*}
  \frac{E_1(z,x_1)}{E_2(z,x_2)}, \quad z\in\C^+
\end{align*}
is of bounded type for some $x_1\in(a_1,b_1)$, $x_2\in(a_2,b_2)$ and $\mu_1=\mu_2$, then the functions $\eta$ and $\kappa$ from the Liouville transform of Theorem~\ref{thmLDuniq} satisfy 
 \begin{align*}
  \eta' = \sqrt{\frac{p_2\circ\eta}{p_1} \frac{|r_1|}{|r_2\circ\eta|}} \quad\text{and} \quad \kappa^2 = \sqrt{\frac{p_2\circ\eta}{p_1} \frac{|r_2\circ\eta|}{|r_1|}} 
 \end{align*}
 almost everywhere on $(\alpha_1,\beta_1)$ with respect to the Lebesgue measure and for the measure coefficients we have
 \begin{align*}
  \varsigma_2\circ\eta = \kappa^{-2} \varsigma_1,  \quad \varrho_2\circ\eta = \kappa^2 \varrho_1 
 \quad\text{and}\quad \chi_2\circ\eta = \kappa^2 \chi_1 - \kappa \kappa^{\qd\prime},
 \end{align*}
 as measures on $(\alpha_1,\beta_1)$.
\end{corollary}

\begin{proof}
 From equation~\eqref{eqnLDphiLT} and the asymptotics in Lemma~\ref{lemLDasympt} we infer that
 \begin{align*}
  \int_{\tilde{x}_1}^{x_1} \sqrt{\frac{|r_1(x)|}{p_1(x)}}dx = \int_{\eta(\tilde{x}_1)}^{\eta(x_1)} \sqrt{\frac{|r_2(x)|}{p_2(x)}} dx, \quad x_1,\,\tilde{x}_1\in(\alpha_1,\beta_1).
 \end{align*} 
 In view of the Banach--Zarecki\u{\i} theorem (see e.g.\ \cite[Chapter~IX; Theorem~4]{nat}, \cite[Theorem~18.25]{hesr})  this shows that $\eta$ is locally absolutely continuous on $(\alpha_1,\beta_1)$ with derivative given as in the claim. More precisely, this follows from an application of~\cite[Chapter~IX; Exercise~13]{nat} and~\cite[Chapter~IX; Theorem~5]{nat}. 
 Furthermore, since $\phi_{1,0}$, $\phi_{2,0}$ are scalar multiples of $w_{1,a}$,  $w_{2,a}$ respectively, we also have
 \begin{align}\tag{$*$}\label{eqnLDlvtranswa}
  w_{1,a}(x_1) = C_a \kappa(x_1) w_{2,a}(\eta(x_1)), \quad x_1\in(\alpha_1,\beta_1)
 \end{align}
 for some constant $C_a\in\R^\times$. In particular, this shows that $\kappa$ is locally absolutely continuous on $(\alpha_1,\beta_1)$. 
 In fact, the substitution rule for Lebesgue--Stieltjes integrals (see e.g.\ \cite{tsr}) shows that 
 \begin{align*}
  w_{2,b}(\eta(x_1)) - w_{2,b}(\eta(\tilde{x}_1)) & = \int_{\eta(\tilde{x}_1)}^{\eta(x_1)} w_{2,b}^\qd \, d\varsigma_2 \\ 
   &  = \int_{\tilde{x}_1}^{x_1} w_{2,b}^\qd\circ\eta\, d\varsigma_2\circ\eta, \quad x_1,\,\tilde{x}_1\in(\alpha_1,\beta_1) 
 \end{align*}
 and hence the function $x_1\mapsto w_{2,b}(\eta(x_1))$ is locally of bounded variation on $(\alpha_1,\beta_1)$.
 Therefore, from~\cite[Chapter~IX; Theorem~5]{nat} we infer that this function is even locally absolutely continuous on $(\alpha_1,\beta_1)$ and hence so is $\kappa$.
 Moreover, in view of Lemma~\ref{lemLDspectrans}, equation~\eqref{eqnLDphiLT} yields 
 \begin{align*}
  \kappa(x_1)^2 & = \frac{\|\delta_{1,x_1}\|_{H^1(a_1,b_1)}^2}{\|\delta_{2,\eta(x_1)}\|_{H^1(a_2,b_2)}^2} = \frac{W(w_{2,b},w_{2,a})}{W(w_{1,b},w_{1,a})} \frac{w_{1,a}(x_1)w_{1,b}(x_1)}{w_{2,a}(\eta(x_1)) w_{2,b}(\eta(x_1))} 
 \end{align*}
 for each $x_1\in(\alpha_1,\beta_1)$. Inserting~\eqref{eqnLDlvtranswa} we get from this equation 
 \begin{align*}
  w_{1,b}(x_1) = C_a^{-1} \frac{W(w_{1,b},w_{1,a})}{W(w_{2,b},w_{2,a})}  \kappa(x_1)w_{2,b}(\eta(x_1)), \quad x_1\in(\alpha_1,\beta_1).
 \end{align*}
 Plugging this expression and equation~\eqref{eqnLDlvtranswa} into the definition of the Wronskian $W(w_{1,b},w_{1,a})$ one obtains 
 \begin{align*}
  1 =  \frac{\kappa(x_1)^2 \eta'(x_1) p_1(x_1)}{p_2(\eta(x_1))}, \quad x_1\in(\alpha_1,\beta_1),
 \end{align*}
 which shows that $\kappa$ is given as in the claim.
 Next, differentiating equation~\eqref{eqnLDphiLT} yields 
 \begin{align*}
  \kappa(x_1) \phi_{1,z}^\qd(x_1) = \kappa^\qd(x_1) \phi_{1,z}(x_1) + \phi_{2,z}^\qd(\eta(x_1)), \quad x_1\in(\alpha_1,\beta_1)
 \end{align*}
 for each $z\in\C$. From this we get for all $\alpha$, $\beta\in(\alpha_1,\beta_1)$ 
 \begin{align*}
  & \int_\alpha^\beta \phi_{1,z} \kappa\, d\chi_1 - z\int_\alpha^\beta \phi_{1,z} \kappa\, d\varrho_1 = \\ 
   & \quad\quad\quad\quad = \int_\alpha^\beta \phi_{1,z}\, d\kappa^\qd + \int_{\alpha}^{\beta} \phi_{1,z} \kappa^{-1} d\chi_2\circ\eta - z\int_{\alpha}^{\beta} \phi_{1,z} \kappa^{-1} d\varrho_2\circ\eta,
 \end{align*} 
 where we used the integration by parts formula~\eqref{eqnLDpartint}, the differential equation and the substitution rule.
 In particular, choosing $z=0$ this shows that the coefficients $\chi_1$ and $\chi_2$ are related as in the claim (note that $\phi_{1,0}$ does not have any zeros).
 Using this relation, one sees from the previous equation that for each $z\in\C^\times$ and $\alpha$, $\beta\in(\alpha_1,\beta_1)$ we actually have
 \begin{align*}
  \int_\alpha^\beta \phi_{1,z} \kappa\, d\varrho_1 = \int_\alpha^\beta \phi_{1,z} \kappa^{-1} d\varrho_2\circ\eta.
 \end{align*}    
 Now since for each $x_1\in(\alpha_1,\beta_1)$ there is some $z\in\C^\times$ such that $\phi_{1,z}(x_1)\not=0$, this shows that the coefficients $\varrho_1$ and $\varrho_2$ are related as in the claim.
\end{proof}

In particular, note that these relations among our measures show that under the assumptions of Corollary~\ref{corLDuniqcoef}, for every $z\in\C$ and each solution $u_2$ of $(\tau_2-z)u=0$, the function
\begin{align*}
 u_1(x_1) = \kappa(x_1) u_2(\eta(x_1)), \quad x_1\in(\alpha_1,\beta_1)
\end{align*}
is a solution of $(\tau_1-z)u=0$ on $(\alpha_1,\beta_1)$. Moreover, linear independence is preserved under this transformation.

In the remaining part of this section we will prove one more inverse uniqueness result, tailor-made to fit the requirements of the isospectral problem of the Camassa--Holm equation. There, we do not want the measures $\varrho_1$ and $\varrho_2$ to necessarily have dense support; hence we can not apply Corollary~\ref{corLDuniqcoef}. However, we will assume that the intervals and the coefficients on the left-hand side of the differential equation are fixed, i.e.
\begin{align*}
 a:= a_1=a_2,\quad b:= b_1=b_2, \quad \varsigma:=\varsigma_1=\varsigma_2 \quad\text{and}\quad \chi:=\chi_1=\chi_2,
\end{align*}
and that $\tau_1$ and $\tau_2$ are in the l.p.\ case at both endpoints.
Another crucial additional assumption we will make for this inverse uniqueness result is that the norms of point evaluations (note that the modified Sobolev spaces are the same for both relations) $\|\delta_c\|_{\Hab}$ are independent of $c\in(a,b)$.
 For example this is the case when $\varsigma$ and $\chi$ are scalar multiples of the Lebesgue measure, as it is the case for the isospectral problem of the Camassa--Holm equation.
Moreover, we suppose that our real entire solutions $\phi_{1,z}$ and $\phi_{2,z}$ coincide at $z=0$, i.e.
\begin{align}\label{eqnLDphizero}
 \phi_{1,0}(x) = \phi_{2,0}(x), \quad x\in(a,b).
\end{align}
As a consequence of these assumptions, the coefficient of the second term on the right-hand side of~\eqref{eqnLDdBemb} in Theorem~\ref{thmLDdBemb} is the same for both problems. Now the weight measure on the right-hand side of our differential equation is uniquely determined by the spectral measure.
In view of application to the isospectral problem of the Camassa--Holm equation we state this result with the assumption that our de Branges functions are of exponential type zero. Of course the same result holds if their quotient is of bounded type in the upper complex half-plane.

\begin{theorem}
Suppose that $E_1(\,\cdot\,,c)$ and $E_2(\,\cdot\,,c)$ are of exponential type zero for some $c\in(a,b)$. If $\mu_1 = \mu_2$, then we have $\varrho_1=\varrho_2$ and $S_1=S_2$.
\end{theorem}

\begin{proof}
 The (remark after the) proof of Lemma~\ref{lemLDeta} shows that there is an increasing continuous bijection $\eta$ from $\Sigma_1$ onto $\Sigma_2$ such that
 \begin{align*}
  B_1(x_1) = B_2(\eta(x_1)), \quad x_1\in\Sigma_1.
 \end{align*}
 Moreover, the proof of Theorem~\ref{thmLDuniq} (see equation~\eqref{eqnLDphiLT}) shows that 
 \begin{align*}
  \phi_{1,z}(x_1) = \kappa(x_1) \phi_{2,z}(\eta(x_1)), \quad z\in\C,~x_1\in\Sigma_1
 \end{align*}
 for some non-zero real function $\kappa$ on $\Sigma_1$.
 In particular, from Lemma~\ref{lemLDspectrans} we infer for each $x_1\in\Sigma_1$ 
 \begin{align*}
  \|\delta_{x_1}\|_{\Hab}^2 = \| \mathcal{F}_1 \delta_{x_1} \|_{\mu_1}^2 = \|\kappa(x_1) \mathcal{F}_2 \delta_{\eta(x_1)} \|_{\mu_1}^2 = \kappa(x_1)^2 \|\delta_{\eta(x_1)} \|_{\Hab}^2
 \end{align*} 
 and hence $\kappa(x_1)^2=1$ in view of our additional assumptions.
 Moreover, Theorem~\ref{thmLDdBemb} shows that $B_1(x_1)$ and $B_2(\eta(x_1))$ actually have the same norm and hence
 \begin{align*}
  \phi_{1,0}^\qd(x_1) \phi_{1,0}(x_1) = K_1(0,0,x_1) = K_2(0,0,\eta(x_1)) = \phi_{1,0}^\qd(\eta(x_1)) \phi_{1,0}(\eta(x_1)).
 \end{align*}
 Now since the function $\phi_{1,0}^\qd\phi_{1,0}$ is strictly increasing on $(a,b)$ we infer that $\eta(x_1)=x_1$, $x_1\in\Sigma_1$ and in particular $\Sigma_1=\Sigma_2$.
 Hence we even have (note that~\eqref{eqnLDphizero} prohibits $\kappa(x_1)=-1$ for some $x_1\in\Sigma_1$)
 \begin{align}\tag{$*$}\label{eqnLDphiequal}
  \phi_{1,z}(x_1) = \phi_{2,z}(x_1), \quad x_1\in\Sigma_1,~z\in\C.
 \end{align} 
 Moreover, if $(\alpha,\beta)$ is a gap of $\Sigma_1$, i.e.\ $\alpha$, $\beta\in\Sigma_1$ but $(\alpha,\beta)\cap\Sigma_1=\emptyset$, then both of this functions are solutions to the same differential equation which coincide on the boundary of the gap. Since their difference is a solution of $\tau_1 u=0$ which vanishes on the boundary of the gap, we infer that~\eqref{eqnLDphiequal} holds for all $x_1$ in the convex hull of $\Sigma_1$ in view monotonicity of the functions in~\eqref{eqnKerTmaxInc}.
Now if $\underline{x}=\inf\Sigma_1>a$, then $\varsigma+\chi$ is infinite near $a$ and for each $z\in\C$ the solutions $\phi_{1,z}$ and $\phi_{2,z}$ are scalar multiples of $w_a$ on $(a,\underline{x})$. Since they are equal in the point $\underline{x}$ we infer that~\eqref{eqnLDphiequal} also holds for all $x_1$ below $\underline{x}$. Similarly, if $\overline{x}=\sup\Sigma_1<b$, then the spectrum of $S_1$ (and hence also of $S_2$) is purely discrete. Indeed, the solutions $\psi_{1,b,z}$, $z\in\C$ of $(\tau_1-z)u=0$ which are equal to $w_b$ near $b$ are real entire and lie in $S_1$ near $b$. Now for each eigenvalue $\lambda\in\R^\times$ the solutions $\phi_{1,\lambda}$ and $\phi_{2,\lambda}$ are scalar multiples of $w_b$ on $(\overline{x},b)$. As before we infer that~\eqref{eqnLDphiequal} holds for $z=\lambda$ and all $x_1\in(a,b)$.
Finally, from the differential equation we get for each $\alpha$, $\beta\in(a,b)$ with $\alpha<\beta$
\begin{align*}
 \lambda \int_\alpha^\beta \phi_{1,\lambda}\, d\varrho_1 & = -\phi_{1,\lambda}^\qd(\beta) + \phi_{1,\lambda}^\qd(\alpha) + \int_\alpha^\beta \phi_{1,\lambda}\, d\chi  = \lambda \int_\alpha^\beta \phi_{2,\lambda}\, d\varrho_2 \\ & = \lambda \int_\alpha^\beta \phi_{1,\lambda}\, d\varrho_2
\end{align*}
for each $\lambda\in\sigma(S_1)$.
But this shows $\varrho_1=\varrho_2$ and hence also $S_1=S_2$.
Hereby note that for each $x\in(a,b)$ there is an eigenvalue $\lambda\in\R$ such that $\phi_{1,\lambda}(x)\not=0$. Indeed, otherwise we had $f(x)=0$ for each $f\in\D_1$, which is not possible unless $\Sigma_1=\emptyset$.
\end{proof}

Note that the condition that the differential expressions are in the l.p.\ case may be relaxed.
For example it is sufficient to assume that $\tau_j$, $j=1,2$ are in the l.p.\ case at $a$ unless $\inf\Sigma_j=a$ and in the l.p.\ case at $b$ unless $\sup\Sigma_j=b$.
The proof therefore is essentially the same.

\appendix
\section{Hilbert spaces of entire functions}\label{appLDbranges}

In this appendix we will briefly summarize some results of de Branges' theory of Hilbert spaces of entire functions as far as it is needed for the proof of our inverse uniqueness theorem. For a detailed discussion we refer to de Branges' book~\cite{dBbook}.
First of all recall that an analytic function $N$ in the upper open complex half-plane $\C^+$ is said to be of bounded type if it can be written as the quotient of two bounded analytic functions.
For such a function the number
\begin{align*}
  \limsup_{y\rightarrow\infty} \frac{\ln|N(\I y)|}{y} \in[-\infty,\infty)
\end{align*}
is referred to as the mean type of $N$.

A de Branges function is an entire function $E$, which satisfies the estimate
\begin{align*}
 |E(z)| > |E(z^\ast)|, \quad z\in\C^+.
\end{align*}
Associated with such a function is a de Branges space $B$. It consists of all entire functions $F$ such that
\begin{align*}
 \int_\R \frac{|F(\lambda)|^2}{|E(\lambda)|^2} d\lambda < \infty
\end{align*}
and such that $F/E$ and $F^\#/E$ are of bounded type in $\C^+$ with non-positive mean type.
Here $F^\#$ is the entire function given by
\begin{align*}
 F^\#(z) = F(z^\ast)^\ast, \quad z\in\C.
\end{align*}
Equipped with the inner product
\begin{align*}
 \dbspr{F}{G} = \frac{1}{\pi} \int_\R \frac{F(\lambda)G(\lambda)^\ast}{|E(\lambda)|^2} d\lambda, \quad F,\,G\in B,
\end{align*}
 the vector space $B$ turns into a reproducing kernel Hilbert space (see~\cite[Theorem~21]{dBbook}).
 For each $\zeta\in\C$, the point evaluation in $\zeta$ can be written as
 \begin{align*}
  F(\zeta) = \dbspr{F}{K(\zeta,\cdot\,)}, \quad F\in B,
 \end{align*}
 where the reproducing kernel $K$ is given by (see~\cite[Theorem~19]{dBbook})
 \begin{align}\label{eqnLDdBrepker}
  K(\zeta,z) = \frac{E(z)E^\#(\zeta^\ast)-E(\zeta^\ast)E^\#(z)}{2\I (\zeta^\ast-z)}, \quad \zeta,\,z\in\C.
 \end{align}
 Note that though there is a multitude of de Branges functions giving rise to the same de Branges space (including norms), the reproducing kernel $K$ is independent of the actual de Branges function. 
 
One of the main results in de Branges' theory is the subspace ordering theorem;~\cite[Theorem~35]{dBbook}. 
 For our application we need to slightly weaken the assumptions of this theorem. 
 In order to state it let $E_1$, $E_2$ be two de Branges functions with no real zeros and $B_1$, $B_2$ be the associated de Branges spaces.

\begin{theorem}\label{thmLDdBordpitt}
Suppose $B_1$, $B_2$ are homeomorphically embedded in $\Lrmu$ for some Borel measure $\mu$ on $\R$.
If $E_1/E_2$ is of bounded type in the open upper complex half-plane, then $B_1$ contains $B_2$ or $B_2$ contains $B_1$.
\end{theorem}
 
\begin{proof}
 If a de Branges space $B$ is homeomorphically embedded in $\Lrmu$, then $B$ equipped with the inner product inherited from $\Lrmu$ is a de Branges space itself. In fact, this is easily verified using the characterization of de Branges spaces in~\cite[Theorem~23]{dBbook}. Hence, without loss of generality we may assume that $B_1$, $B_2$ are isometrically embedded in $\Lrmu$ and thus apply~\cite[Theorem~35]{dBbook}. Therefore, also note that $F_1/F_2$ is of bounded type in the upper complex half-plane for all $F_1\in B_1$, $F_2\in B_2$ and hence so is the quotient of any corresponding de Branges functions. 
\end{proof}
  
 Note that the isometric embedding in~\cite[Theorem~35]{dBbook} is only needed to deduce that the smaller space is actually a de Branges subspace of the larger one. The inclusion part is valid under much more general assumptions; see~\cite[Theorem~5]{pitt0} or~\cite[Theorem~3.5]{pitt}.
   
 Adapting the proof of~\cite[Theorem~35]{dBbook}, one gets a version of de Branges' ordering theorem, where the bounded type condition is replaced by the assumption that the functions $E_1$, $E_2$ are of exponential type zero.
 Actually this has been done in~\cite{kotani86} with the spaces $B_1$, $B_2$ being isometrically embedded in some $\Lrmu$. Again this latter assumption can be weakened.
 
\begin{theorem}\label{thmLDdBordkotani}
Suppose $B_1$, $B_2$ are homeomorphically embedded in $\Lrmu$ for some Borel measure $\mu$ on $\R$.
If $E_1$, $E_2$ are of exponential type zero, then $B_1$ contains $B_2$ or $B_2$ contains $B_1$.
\end{theorem}

\begin{proof}
As in the proof of Theorem~\ref{thmLDdBordpitt}, the claim can be reduced to the case where the de Branges spaces are isometrically embedded in $\Lrmu$. Therefore, also note that a de Branges function is of exponential type zero if and only if all functions in the corresponding de Branges space are (see e.g.\ \cite[Theorem~3.4]{kawo}).
\end{proof}

\bigskip
\noindent
{\bf Acknowledgments.}
I thank Harald Woracek for helpful discussions and hints with respect to the literature.


\begin{thebibliography}{XX}
\bibitem{arens}
R.\ Arens, {\em Operational calculus of linear relations}, Pacific J.\ Math.\ {\bf 11} (1961), 9--23.
\bibitem{bss}
R.\ Beals, D.\ H.\ Sattinger and J.\ Szmigielski, {\em Multipeakons and the classical moment problem}, Adv.\ Math.\ {\bf 154} (2000), no.~2, 229--257.
\bibitem{ben1}
C.\ Bennewitz, {\em A Paley--Wiener theorem with applications to inverse spectral theory}, in {\em Advances in differential equations and mathematical physics}, 21--31, Contemp.\ Math. {\bf 327}, Amer.\ Math.\ Soc., Providence, RI, 2003.
\bibitem{ben2}
C.\ Bennewitz, {\em Spectral asymptotics for Sturm--Liouville equations},
Proc.\ London Math.\ Soc.\ (3) {\bf 59} (1989), no.~2, 294--338.
\bibitem{ben3}
C.\ Bennewitz, {\em On the spectral problem associated with the Camassa--Holm equation},
J.\ Nonlinear Math.\ Phys.\ {\bf 11} (2004), no.~4, 422--434.
\bibitem{benbro}
C.\ Bennewitz and B.\ M.\ Brown, {\em A limit point criterion with applications to nonselfadjoint equations},
J.\ Comput.\ Appl.\ Math.\ {\bf 148} (2002), no.~1, 257--265.
\bibitem{bebrwe}
C.\ Bennewitz, B.\ M.\ Brown and R.\ Weikard, {\em Inverse spectral and scattering theory for the half-line left-definite Sturm--Liouville problem},
SIAM J.\ Math.\ Anal.\ {\bf 40} (2008/09), no.~5, 2105--2131.
\bibitem{beco}
A.\ Bressan and A.\ Constantin, {\em Global conservative solutions of the Camassa--Holm equation}, Arch.\ Ration.\ Mech.\ Anal.\ {\bf 183} (2007), no.~2, 215--239.
\bibitem{const}
A.\ Constantin, {\em On the scattering problem for the Camassa--Holm equation}, R.\ Soc.\ Lond.\ Proc.\ Ser.\ A Math.\ Phys.\ Eng.\ Sci.\ {\bf 457} (2001), no.~2008, 953--970.
\bibitem{coes}
A.\ Constantin and J.\ Escher, {\em Global existence and blow-up for a shallow water equation}, Ann.\ Scuola Norm.\ Sup.\ Pisa Cl.\ Sci.\ (4) {\bf 26} (1998), no.~2, 303--328.
\bibitem{como}
A.\ Constantin and L.\ Molinet, {\em Global weak solutions for a shallow water equation}, Comm.\ Math.\ Phys.\ {\bf 211} (2000), no.~1, 45--61.
\bibitem{cost}
A.\ Constantin and W.\ Strauss, {\em Stability of peakons}, Comm.\ Pure Appl.\ Math.\ {\bf 53} (2000), no.~5, 603--610.
\bibitem{cross}
R.\ Cross, {\em Multivalued Linear Operators}, Monographs and Textbooks in Pure and Applied Mathematics {\bf 213}, Marcel Dekker, New York, 1998.
\bibitem{dBbook}
L.\ de Branges, {\em Hilbert spaces of entire functions}, Prentice-Hall, Inc., Englewood Cliffs, N.J.\ 1968.
\bibitem{dijksnoo}
A.\ Dijksma and H.\ S.\ V.\ de Snoo, {\em Self-adjoint extensions of symmetric subspaces}, Pacific J.\ Math.\ {\bf 54} (1974), 71--100.
\bibitem{dijksnoo2}
A.\ Dijksma and H.\ S.\ V.\ de Snoo, {\em Symmetric and selfadjoint relations in Kre\u{\i}n spaces I}, in {\em Operators in indefinite metric spaces, scattering theory and other topics} (Bucharest, 1985), 145--166, Oper.\ Theory Adv.\ Appl.\ {\bf 24}, Birkh\"{a}user, Basel, 1987. 
\bibitem{IsospecCH}
J.\ Eckhardt and G.\ Teschl, {\em On the isospectral problem of the dispersionless Camassa--Holm equation}, in preparation.
\bibitem{measureSL}
J.\ Eckhardt and G.\ Teschl, {\em Sturm--Liouville operators with measure-valued coefficients}, J.\ Anal.\ Math.\ (to appear).
\bibitem{tsr}
N.\ Falkner and G.\ Teschl, {\em On the substitution rule for Lebesgue--Stieltjes integrals}, \arxiv{1104.1422}.
\bibitem{geszin}
F.\ Gesztesy and M.\ Zinchenko, {\em On spectral theory for Schr\"{o}dinger operators with strongly singular potentials}, Math.\ Nachr.\ {\bf 279} (2006), no.~9-10, 1041--1082.
\bibitem{haase}
M.\ Haase, {\em The Functional Calculus for Sectorial Operators}, Operator Theory: Advances and Applications {\bf 169}, Birkh\"auser, Basel, 2006.
\bibitem{hesr}
E.\ Hewitt and K.\ Stromberg, {\em Real and Abstract Analysis}, Springer, New York, 1965.
\bibitem{hora}
H.\ Holden and X.\ Raynaud, {\em Global conservative solutions of the Camassa--Holm equation---a Lagrangian point of view}, Comm.\ Partial Differential Equations {\bf 32} (2007), no.~10-12, 1511--1549.
\bibitem{kawo}
M.\ Kaltenb\"{a}ck and H.\ Woracek, {\em De Branges spaces of exponential type: general theory of growth}, Acta Sci.\ Math.\ (Szeged) {\bf 71} (2005), no.~1-2, 231--284.
\bibitem{kst2}
A.\ Kostenko, A.\ Sakhnovich and G.\ Teschl, {\em Weyl--Titchmarsh theory for Schr\"odinger operators with strongly singular potentials}, Int.\ Math.\ Res.\ Not.\ {\bf 2011}, Art.\ ID rnr065, 49pp.
\bibitem{kotani86}
S.\ Kotani, {\em A remark to the ordering theorem of L.\ de Branges}, J.\ Math.\ Kyoto Univ.\ {\bf 16} (1976), no.~3, 665--674.
\bibitem{lev}
B.\ Ya.\ Levin, {\em Lectures on Entire Functions}, Transl.\ Math.\ Mon.\ {\bf 150}, Amer.\ Math.\ Soc., Providence, RI, 1996.
\bibitem{nat}
I.\ P.\ Natanson, {\em Theory of functions of a real variable}, Frederick Ungar Publishing Co., New York, 1955.
\bibitem{pitt0}
L.\ D.\ Pitt, {\em Weighted $L^p$ closure theorems for spaces of entire functions}, Israel J.\ Math.\ {\bf 24} (1976), no.~2, 94--118.
\bibitem{pitt}
L.\ D.\ Pitt, {\em A general approach to approximation problems of the Bernstein type}, Adv.\ in Math.\ {\bf 49} (1983), no.~3, 264--299.
\bibitem{rosrov}
M.\ Rosenblum and J.\ Rovnyak, {\em Topics in Hardy classes and univalent functions}, Birkh\"{a}user Verlag, Basel, 1994.
\end{thebibliography}
\end{document}